\documentclass{amsart}
\usepackage{amsmath}
\usepackage{hyperref}
\usepackage{color}

\allowdisplaybreaks
\DeclareSymbolFont{bbold}{U}{bbold}{m}{n}
\DeclareSymbolFontAlphabet{\mathbbold}{bbold}
\newcommand{\ind}{\mathbbold{1}}

\numberwithin{equation}{section}

\theoremstyle{plain} \newtheorem{theorem}{Theorem}[section]
\theoremstyle{plain} 
\theoremstyle{plain} \newtheorem{corollary}[theorem]{Corollary}
\theoremstyle{plain} \newtheorem{proposition}[theorem]{Proposition}
\theoremstyle{remark} \newtheorem{remark}[theorem]{Remark}
\theoremstyle{definition} \newtheorem{example}[theorem]{Example}
\theoremstyle{definition} \newtheorem{definition}[theorem]{Definition}
\theoremstyle{remark} 

\newcommand{ \R}{ \mathbb R }

\newcommand{\N}{ \mathbb N}
\renewcommand{\Pr}{ \mathbb P}

\newcommand{\bZ}{ \mathbb Z}

\newcommand{\cB}{\mathcal{B}}
\newcommand {\cG}{\mathcal{G}}
\newcommand {\cC}{\mathcal{C}}
\newcommand {\cD}{\mathcal{D}}

\newcommand{\sgn}{\operatorname{sgn}}

\newcommand{\bR}{\mathbb R}
\newcommand{\bP}{\mathbb P}
\newcommand{\bQ}{\mathbb Q}
\newcommand{\cF}{\mathcal F}

\def\be{\begin{equation}}
\def\ee{\end{equation}}

\begin{document}
\title[Conditioned Markov processes]{Markov processes conditioned on their location at large exponential times}

\author[S.N. Evans]{Steven N. Evans}
\thanks{S.N.E. was supported in part by NSF grant DMS-0907630, NSF grant DMS-1512933, and NIH grant 1R01GM109454-01}
\address{Department of Statistics \#3860\\
 367 Evans Hall \\
 University of California \\
  Berkeley, CA  94720-3860 \\
   USA}
\email{evans@stat.berkeley.edu}
\author[A. Hening]{Alexandru Hening }
\thanks{A.H. was supported by EPSRC grant EP/K034316/1}
\address{ University of Oxford\\
 Department of Statistics \\
 24-29 St Giles \\
 Oxford, OX1 3LB\\
 United Kingdom}
 \email{al.hening@gmail.com}

\begin{abstract}
Suppose that $(X_t)_{t \ge 0}$ is a one-dimensional Brownian motion with negative drift $-\mu$.
It is possible to make sense of conditioning this process to be in the
state $0$ at an independent exponential random time and if we kill the conditioned process
at the exponential time the resulting process is Markov.  If we let the rate parameter
of the random time go to $0$, then the limit of the killed Markov process evolves like
$X$ conditioned to hit $0$, after which time it behaves as $X$ killed at the last time
$X$ visits $0$.  Equivalently, the limit process has the dynamics of the killed ``bang--bang''
Brownian motion that evolves like Brownian motion with positive drift $+\mu$ when it is negative, like
Brownian motion with negative drift $-\mu$ when it is positive, and is killed according to the local time
spent at $0$.


An extension of this result holds in great generality for Borel right processes conditioned to be
in some state $a$ at an exponential random time, at which time they are killed.  Our proofs involve understanding
the Campbell measures associated with local times, 
the use of excursion theory, and  the development of a suitable analogue of the 
``bang--bang'' construction for general Markov processes.


As examples, we consider the special case when the transient Borel right process is a one-dimensional diffusion.  Characterizing
the limiting conditioned and killed process via its infinitesimal generator leads to an investigation of the $h$-transforms
of transient one-dimensional diffusion processes that goes beyond what is known and is of independent interest.
\end{abstract}

\maketitle

\tableofcontents

\section{Introduction}\label{s:intro}
A basic phenomenon that lies at the core of the theory of continuous time Markov processes is the fact that sometimes goes by the name of
``competing exponentials'':  if $\zeta$ and $\xi$ are independent random exponential random variables with respective rate parameters $\lambda$ and $\mu$, then $\bP\{\zeta < \xi\} = \frac{\lambda}{\lambda + \mu}$
and conditional on the event $\{\zeta < \xi\}$ the random variables
$\zeta$ and $\xi - \zeta$ are independent with exponential distributions
that have rate parameters $\lambda + \mu$ and $\mu$.

Letting $\lambda \downarrow 0$, we see that asymptotically the conditional distribution of $(\zeta, \xi-\zeta)$ given $\{\zeta < \xi\}$  is that of
a pair of independent exponential random variables with the same rate parameter $\mu$.

More generally, if $\zeta$ and $\xi$ are independent with $\zeta$ having an
exponential distribution with rate parameter $\lambda$ and $\xi$ is now
an arbitrary nonnegative random variable with a finite nonzero expectation, then
\[
\lim_{\lambda \downarrow 0} \bP\{\xi \in dx \, | \, \zeta < \xi\}
=
\frac{x \bP\{\xi \in dx\}}{\bP[\xi]}
\]
and
\[
\lim_{\lambda \downarrow 0} \bP\{\zeta \in dz \, | \, \zeta < \xi, \, \xi = x\}
=
\frac{\ind\{z < x\} \, dz}{x}.
\]
In particular,
\[
\lim_{\lambda \downarrow 0} \bP\{\zeta \in dz \, | \, \zeta < \xi\}
=
\frac{\bP\{\xi > z\} \, dz}{\bP[\xi]}.
\]
If we let $M$ be the random measure that is the restriction of
Lebesgue measure to the interval $[0,\xi)$, then
one way of expressing the last set of results is that
\[
\lim_{\lambda \downarrow 0} \bP\{\xi \in dx, \, \zeta \in dz \, | \, \zeta < \xi\}
=
\frac{\bP[\ind\{\xi \in dx\} \, M(dz)]}{\bP[M(\bR_+)]}.
\]

The probability measure on $\Omega \times \bR_+$ that assigns mass
\[
\frac{\bP[\ind_A \, M(B)]}{\bP[M(\bR_+)]}
\]
to the set $A \times B$ is called the {\em Campbell measure} associated
with the random measure $M$.  In this paper we will be interested in
Campbell measures in the case where $M$ is the local time at some
state $a$ for a transient Markov process.   As one might expect from the above 
calculations, the Campbell measure
may be interpreted as describing the limit as $\lambda \downarrow 0$
of the joint distribution of the Markov process and the independent exponential $\zeta$ conditional
on the event that the Markov process is in the state $a$ at time $\zeta$.

We next present a simple example that motivates our work and doesn't require any sophistication
in describing what we mean by conditioning a Markov process to be in a given state at an
independent exponential time because in this example the  event on which we are conditioning has positive probability.

\begin{example}\label{ex:RW_1}
Suppose that $(X_t)_{t \ge 0}$ is the continuous-time simple random walk on the integers
that jumps to the states $x-1$ and $x+1$ with respective rates $\alpha$ and $\beta$ when
it is in state $x \in \bZ$.  Suppose further that $\zeta$ is an independent nonnegative random variable that
has the exponential distribution with rate $\lambda > 0$.  Let $(X_t^\lambda)_{t \ge 0}$
be the process that is obtained by conditioning on the event $\{X_\zeta = 0\}$ and killing
the resulting process at the time $\zeta$.  Then,  $(X_t^\lambda)_{t \ge 0}$ is a Markov process with
\[
\bP^x\{X_t^\lambda = y\}
=
\frac{\bP^x\{X_t = y, \, \zeta > t, \, X_\zeta = 0\}}{\bP^x\{ X_\zeta = 0\}}
=
\frac{\bP^x[\ind\{X_t = y\} e^{-\lambda t} \lambda r_\lambda(y,0)]}{\lambda r_\lambda(x,0)},
\]
where $r_\lambda(u,v) := \int_0^\infty e^{-\lambda t} \bP^u\{X_t = v\} \, dt$.

Assume that $\alpha < \beta$.  Note that
$\lim_{\lambda \downarrow 0} r_\lambda(u,v) = r_0(u,v)$,
where $r_0(u,v) := \int_0^\infty \bP^u\{X_t = v\} \, dt$
satisfies
\[
r_0(u,v)
=
\begin{cases}
r_0(v,v) = r_0(0,0),& \quad  \text{if $u \le v$}, \\
\left(\frac{\alpha}{\beta}\right)^{u-v} r_0(v,v) = \left(\frac{\alpha}{\beta}\right)^{u-v} r_0(0,0),
& \quad  \text{if $u > v$}.
\end{cases}
\]
Therefore, as $\lambda \downarrow 0$ the Markov process $(X_t^\lambda)_{t \ge 0}$
converges to a Markov process $(X_t^0)_{t \ge 0}$ with
\[
\bP^x\{X_t^0 = y\}
=
\lim_{\lambda \downarrow 0} \bP^x\{X_t^\lambda = y\}
=
\frac{\bP^x\{X_t = y\} \left(\frac{\alpha}{\beta}\right)^{y_+}}{\left(\frac{\alpha}{\beta}\right)^{x_+}},
\]
where $x_+:=\max(x,0)$. Let $\cG$ be the infinitesimal generator of $(X_t^0)_{t \ge 0}$.  For a function $f:\bZ \to \bR$ we have
\[
\begin{split}
\cG f(x)
& =
\frac{
\alpha f(x-1) \left(\frac{\alpha}{\beta}\right)^{(x-1)_+}
+
\beta f(x+1) \left(\frac{\alpha}{\beta}\right)^{(x+1)_+}
- (\alpha + \beta) f(x) \left(\frac{\alpha}{\beta}\right)^{x_+}
}
{
\left(\frac{\alpha}{\beta}\right)^{x_+}
} \\
& =
\begin{cases}
\beta f(x-1) + \alpha f(x+1) -  (\alpha + \beta) f(x),& \quad  \text{if $x > 0$},\\
\alpha f(x-1) + \alpha f(x+1) -  (\alpha + \beta) f(x),& \quad  \text{if $x = 0$},\\
\alpha f(x-1) + \beta f(x+1) -  (\alpha + \beta) f(x),& \quad  \text{if $x < 0$}.
\end{cases}
\end{split}
\]
In other words, $(X_t^0)_{t \ge 0}$ is obtained by taking the Markov process $(Y_t)_{t \ge 0}$
with the following jump rates
\begin{itemize}
\item
$x \to x-1$ at rate $\alpha$ when $x < 0$,
\item
$x \to x+1$ at rate $\beta$ when $x < 0$,
\item
$x \to x-1$ at rate $\alpha$ when $x = 0$,
\item
$x \to x+1$ at rate $\alpha$ when $x = 0$,
\item
$x \to x-1$ at rate $\beta$ when $x > 0$,
\item
$x \to x+1$ at rate $\alpha$ when $x > 0$,
\end{itemize}
and killing this process at rate $\beta - \alpha$ when it is in state $0$.
The process $(Y_t)_{t \ge 0}$ is pushed upwards when it is negative and downwards when
it is positive and is analogous to the ``bang--bang Brownian motion'' or ``Brownian motion with alternating drift''
of \cite{GS00, Bor02, RY09}
that, for some $\mu > 0$, evolves like Brownian motion with drift $+\mu$ when it is negative and
like Brownian motion with drift $-\mu$ when it is positive. 

Note that $(X_t)_{t \ge 0}$ started at $X_0=+1$ hits the state $0$
with probability $\frac{\alpha}{\beta}$ and wanders off to $+\infty$ without
hitting the state $0$ with probability
$\frac{\beta-\alpha}{\beta}$, and that $(X_t)_{t \ge 0}$ started at $X_0=+1$, conditioned
to hit the state $0$ and killed when it does so evolves like the process $(Y_t)_{t \ge 0}$
started at $Y_0=+1$ and killed when it hits the state $0$.

Let $(W^{n,-})_{n \in \N}$ (respectively, $(W^{n,+})_{n \in \N}$)
be an i.i.d. sequence of killed paths with common distribution that
of the Markov process that starts in the state $0$, jumps at rate $\alpha$
to the state $-1$ (respectively, $+1$), and then evolves like the process $(Y_t)_{t \ge 0}$
started at $-1$ (respectively, $+1$) and killed when it hits the state $0$.
Define $(W^{n,\infty})_{n \in \N}$ to be an i.i.d. sequence of paths with common distribution that
of the Markov process that starts in the state $0$,
jumps to the state $+1$ at rate $\beta - \alpha$, and thereafter evolves like
the process $(X_t)_{t \ge 0}$ started at $+1$ and conditioned never to hit $0$.
Suppose further that these three sequences are independent.
Put $T_n^- := \inf\{t \ge 0 : W^{n,-}_t \ne 0\}$ and define $T_n^+$ and $T_n^\infty$ similarly.
Set
\[
W^n :=
\begin{cases}
W^{n,-},& \quad  \text{if $T_n^- = T_n^- \wedge T_n^+ \wedge T_n^\infty$},\\
W^{n,+},& \quad  \text{if $T_n^+ = T_n^- \wedge T_n^+ \wedge T_n^\infty$},\\
W^{n,\infty},& \quad  \text{if $T_n^\infty = T_n^- \wedge T_n^+ \wedge T_n^\infty$},
\end{cases}
\]
and
\[
\tilde W^n :=
\begin{cases}
W^{n,-},& \quad  \text{if $T_n^- = T_n^- \wedge T_n^+$},\\
W^{n,+},& \quad  \text{if $T_n^+ = T_n^- \wedge T_n^+$}.\\
\end{cases}
\]
We see that
$(X_t)_{t \ge 0}$ starting at $X_0 = 0$ is obtained by concatenating the excursion paths
$W^1, W^2, \ldots, W^N$, where
$N := \inf\{n : T_n^\infty = T_n^- \wedge T_n^+ \wedge T_n^\infty\}$,
and
$(Y_t)_{t \ge 0}$ starting at $Y_0 = 0$ is obtained by concatenating the excursion paths
$\tilde W^1, \tilde W^2, \ldots$
Observe that $N$ takes the value $n$ with probability
$\left(\frac{2 \alpha}{\alpha + \beta}\right)^{n-1} \frac{\beta - \alpha}{\alpha + \beta}.$

Let $(W^{n,\pm})_{n \in \N}$ be i.i.d. with $W^{n,\pm}$ distributed as $W^n$
conditional on $W^n$ being either $W^{n,-}$ or $W^{n,+}$ (that is, conditional on
$T_n^\infty > T_n^- \wedge T_n^+ \wedge T_n^\infty$).  Note that $W^{n, \pm}$
starts in the state $0$, jumps at rate $\alpha + \beta$, jumps to
state $-1$ (respectively, $+1$) with probability $\frac{1}{2}$, and thereafter
evolves like $(Y_t)_{t \ge 0}$ killed when it first hits the state $0$.
On the other hand, $\tilde W^n$ starts in the state $0$, jumps at rate $2 \alpha$, jumps to
state $-1$ (respectively, $+1$) with probability $\frac{1}{2}$, and thereafter
evolves like $(Y_t)_{t \ge 0}$ killed when it first hits the state $0$.

It follows that if we kill
the process $(Y_t)_{t \ge 0}$ at rate $\beta-\alpha$ when it is in the state $0$, then
the resulting process has the same distribution as the concatenation of the paths
$W^{1,\pm}, \ldots, W^{N'-1,\pm}$, where $N'$ is an independent random variable
that takes the value $n$ with probability
$\left(\frac{2 \alpha}{\alpha + \beta}\right)^{n-1} \frac{\beta - \alpha}{\alpha + \beta}$,
concatenated with a final independent path that is constant at $0$ and is killed at rate $\alpha + \beta$.

Let $\rho_n := T_n^- \wedge T_n^+ \wedge T_n^\infty$ be the amount of time that $W^n$ spends in the state $0$
(so that $\rho_n$ has an exponential distribution with rate $\alpha + \beta$),
$\sigma_n$ be the amount of time that $W^n$ spends in states other than $0$,
and $(\tau_n)_{n \in \N}$ be a sequence of i.i.d. random variables
with a common distribution that is exponential with rate $\lambda$.
We see that $(X_t)_{0 \le t < \zeta}$ is obtained by concatenating the paths
$\hat W^1, \ldots \hat W^M$, where $\hat W^n$ is $W^n$ killed at $\tau_n \wedge (\rho_n + \sigma_n)$ and
$M := \inf\{n : \tau_n < \rho_n + \sigma_n\} \le N$.

Write $\rho_n^\pm$ for the amount of time that $W^{n,\pm}$ spends in the state $0$
(so that $\rho_n^\pm$ has an exponential distribution with rate $\alpha + \beta$) and
$\sigma_n^\pm$ for the amount of time that $W^{n,\pm}$ spends in the states other
than $0$.
Then, 
\[
\begin{split}
& \bP\{W^1 \in dw^1, \ldots, W^{m-1} \in dw^{m-1}, \, \tau_m < \rho_m, \, \tau_m \in dt, M=m\} \\
& \quad  =
\left(\frac{2 \alpha}{\alpha + \beta}\right)^{m-1} \prod_{k=1}^{m-1}
\bP\left[e^{-\lambda(\rho_k^\pm + \sigma_k^\pm)} \ind\{W^{k,\pm} \in dw^k\}\right]
\frac{\lambda}{\lambda +  \alpha + \beta} \\
& \qquad ( \alpha + \beta + \lambda) \, e^{-( \alpha + \beta + \lambda) t} \, dt. \\
\end{split}
\]
Therefore
\[
\begin{split}
& \lim_{\lambda \downarrow 0}
\bP\{W^1 \in dw^1, \ldots, W^{m-1} \in dw^{m-1}, \, \tau_m < \rho_m, \, \tau_m \in dt, M=m
\, | \,
\tau_M < \rho_M\} \\
& \quad =
\left(\frac{2 \alpha}{\alpha + \beta}\right)^{m-1} \frac{\beta-\alpha}{\alpha+\beta}
\prod_{k=1}^{m-1}
\bP\left[\ind\{W^{k,\pm} \in dw^k\}\right]
e^{-( \alpha + \beta) t} \, dt \\
\end{split}
\] 
so that $(X_t)_{0 \le t < \zeta}$ started at $X_0 = 0$
and conditioned on $\{X_\zeta = 0\}$
converges in distribution as $\lambda \downarrow 0$ to a process 
that is distributed as the concatenation of
$W^{1,\pm}, \ldots, W^{N^* - 1,\pm}$, where $N^*$ is an independent random variable
with the same distribution as $N$, concatenated with a final independent path that is constant
at $0$ and killed at rate $ \alpha + \beta$.

Hence $(X_t)_{0 \le t < \zeta}$ started at $X_0 = 0$ and conditioned on $\{X_\zeta = 0\}$
has the same distribution in the limit $\lambda \downarrow 0$
as $(X_t)_{t \ge 0}$ killed at the time the process leaves the state $0$ for the last time and,
moreover, this distribution is the same as that of  $(Y_t)_{t \ge 0}$ started at $Y_0 = 0$
and killed at rate $\beta - \alpha$ in state $0$.
\end{example}

Our aim in this paper is to show that results analogous to those obtained for the
continuous--time simple random walk in Example \ref{ex:RW_1} hold in great generality; specifically,
if we condition a transient Borel right process to be in a fixed regular 
state $a$ at some independent exponential time $\zeta$,
kill the process at $\zeta$, and let the rate parameter of $\zeta$ go to $0$, then the Borel right process looks
like a certain recurrent Borel right process process that is killed according to an appropriate mechanism when it is
in the state $a$.  Moreover, the limit of the killed Borel right process evolves like the original process
conditioned to hit the point $a$ after which it behaves as the original process until it is killed
at the last time the original process leaves the state $a$.

We will, of course, require certain conditions.  The transient Borel right process  must have positive probability of hitting
the state $a$ from any starting point and we will also need the existence of a suitable local time at $a$
in order to make sense of the idea of conditioning the Borel right process
on being in state $a$ at time $\zeta$ when the Lebesgue measure of the set of times that the process
spends in $a$ is almost surely zero (and so the event on which we are conditioning has probability zero).

The paper is organized as follows.

The {\em Campbell measure} associated with a random measure $M$ such that $0 < \bP[M(\bR_+)] < \infty$
 is the probability measure $\bar \bP$ on $\Omega \times \bR_+$
given by 
\[
\bar \bP(A \times B) := \frac{\bP[\ind_A M(B)]}{\bP[M(\bR_+)]}.
\]
In Section \ref{s:conditioning_exponential} we  establish the connection between Campbell measures and the limit
as $\lambda \downarrow 0$ of conditioning a random set to contain an independent exponential
random variable with rate parameter $\lambda$.



We start discussing Borel right processes in Section \ref{s:general}.
For such a process $X$ and $a\in E$ let $T_a := \inf\{t > 0 : X_t = a\}$
and $K_a := \sup\{t \ge 0 : X_t = a\}$ be the first and last hitting times of $a$, 
where we adopt
the usual conventions that $\inf \emptyset = +\infty$ and $\sup \emptyset = 0$.
Our starting point is the following result which we prove in Section~\ref{s:general}.
Here $\xi: \Omega \times \bR_+ \to \bR_+$ is given by $\xi(\omega,t) = t$.

\begin{theorem}\label{t:main_1}
Let $X$ be a Borel right process
with Lusin state space $E$.  Suppose that $a \in E$ is such that  
\begin{itemize}
\item $\bP^a\{K_a < \infty\} = 1$, 
\item $\bP^a\{T_a = 0\} = 1$,
\item $\bP^x\{T_a < \infty\} > 0$ for all $x \in E$.
\end{itemize}
If $Z$ is a nonnegative $\cF_t$-measurable random variable for some $t \ge 0$, then
\[
\bar \bP^x[Z \ind\{\xi > t\}]
=
\frac{1}{\bP^x\{T_a<\infty\}} \bP^x\left[Z \bP^{X_t}\{T_a<\infty\}\right],
\]
where $\bar \bP^x$ is the Campbell measure associated with the local time of $X$ at $a$.
Moreover,
the distribution of $(X_t)_{0 \le t < \xi}$ under the Campbell measure $\bar \bP^x$
is the same as the distribution of $(X_t)_{0 \le t < K_a}$ under $\bP^x$ conditional
on $\{T_a < \infty\}$.
\end{theorem}

This theorem says heuristically that if $\kappa$ is an independent random variable that has an exponential distribution with rate parameter $\lambda$, then the distribution of $(X_t)_{0\leq t<\kappa}$ under $\Pr^x$ conditional on the event $\{X_\kappa=a\}$ converges as $\lambda\downarrow 0$ to the distribution of $(X_t)_{0\leq t < K_a}$ under $\Pr^x$ conditional on the event $\{T_a<\infty\}$.

We discuss excessive functions and general Doob $h$-transforms for Borel right processes
in Section \ref{s:doob_general}.

In Section \ref{s:bangbang} we construct a generalization of the bang-bang Brownian 
motion or Brownian motion with alternating drift \cite{GS00, Bor02, RY09} 
in which Brownian motion is
replaced by a general Borel right process $X$ with a regular state $a$. 
We use the notion of resurrected Markov processes 
(see \cite{meyer75, Fitz91} and Example 5.14 from \cite{FG06}).
The general bang-bang process is a Markov process that behaves 
like $X$ conditioned to hit $a$ 
until it hits $a$ and then looks like a process started at $a$ that can be built
from the same Poisson point process of excursions from $a$ as $X$ except that
only excursions of finite length are used (so the process keeps returning to
$a$).

As a consequence of these constructions we get the following result for general Borel right processes
which we prove in Section~\ref{s:bangbang_2}.

\begin{theorem}\label{t:main_2}
Let $X$ be a Borel right process with a Lusin state space $E$ and let $a\in E$. 
Suppose that $a \in E$ is such that
\begin{itemize}
\item $\bP^a\{K_a < \infty\} = 1$,
\item $\Pr^a\{T_a=0\}=1$,
\item $\Pr^x\{T_a<\infty\}>0$ for all $x\in E$.
\end{itemize}
Suppose, moreover, that
the resolvent $(R_\lambda)_{\lambda > 0}$ of $X$ has a density with respect to a measure $m$.

Then for any $x \in E$ the distribution of $(X_t)_{0 \le t < \xi}$ 
under the Campbell measure $\bar \bP^x$
associated with the local time at $a$ is that of the recurrent Borel right process $X^b$ 
constructed in Section \ref{s:bangbang_2} killed when the local time of $X^b$ at $a$ exceeds an 
independent exponential random variable with rate parameter equal to the It\^o excursion measure mass of the infinite excursions of $X$ from $a$.
\end{theorem}

Sections \ref{s:doob} and \ref{s:doob_generator} contain a 
 study of $h$-transforms for general transient one-dimensional diffusions. 
 After recalling the characteristics of a one-dimensional diffusion 
 -- the scale function, speed measure, and killing measure 
 -- we show in Theorem \ref{t:h_diffusion_chars} how these characteristics 
 change under an $h$-transform. This fact is well-known in the folklore, 
 but we present a proof because we were not able to find one in the literature 
 that treats the general case we need. We then characterize the 
 generator of the $h$-transformed diffusion.

Section \ref{s:prop} considers the bang-bang construction for the special
case of one-dimensional diffusions and Section \ref{s:fixedpoint} investigates
the generator of the $h$-transformed process of Theorem \ref{t:main_1} when
the process $X$ is a one-dimensional diffusion. 
In Section \ref{s:another_conditioning} we briefly discuss a different way of conditioning a Markov process to be in a fixed state at a large random time.

 %
%
%

\section{Campbell measures}\label{s:conditioning_exponential}

Suppose that on some probability space $(\Omega, \cF, \bP)$
we have a random set $S \subset \bR_+$ such that
$0 < \bP[|S|] < \infty$, where $|\cdot|$ is Lebesgue measure. 
Let $\nu_\lambda$ be the exponential distribution on
$\bR_+$ with rate $\lambda$.
Define $\xi$ to be the canonical random variable on
$(\bR_+, \cB(\bR_+), \nu_\lambda)$.
With the usual abuse of notation, we can think of $S$ and $\xi$ as being
defined on
$(\Omega \times \bR_+, \cF \otimes \cB(\bR_+), \bP \otimes \nu_\lambda)$.
Define the probability measure $\bar \bP_\lambda$ on
$(\Omega \times \bR_+, \cF \otimes \cB(\bR_+))$ by
\[
\bar \bP_\lambda(A \times B)
:=
\frac{
\bP \otimes \nu_\lambda\{(\omega,t) : \omega \in A, \, t \in B \cap S(\omega)\}
}
{
\bP \otimes \nu_\lambda\{(\omega,t) : t \in S(\omega)\}
};
\]
that is, $\bar \bP_\lambda$ is $\bP \otimes \nu_\lambda$
conditioned on the event $\{\xi \in S\}$.  Note that
\[
\bar \bP_\lambda(A \times B)
=
\frac{
\bP\left[\ind_A \int_{B \cap S} \lambda e^{-\lambda t} \, dt\right]
}
{
\bP\left[\int_S \lambda e^{-\lambda t} \, dt\right]
}.
\]
Letting $\lambda \downarrow 0$ we get the probability measure
\[
\bar \bP(A \times B)
:=
\frac{
\bP[\ind_A |B \cap S|]
}
{
\bP[|S|]
}
=
\frac{
\bP[\ind_A M(B)]
}
{
\bP[M(\bR_+)]
},
\]
where $M$ is the random measure given by
$M(C) := |C \cap S|$.   We can think of the probability measure
$\bar \bP$ as describing what happens asymptotically when we condition
on $S$ containing a large exponential time.

More generally, if $M$ is an arbitrary random measure with
$0 < \bP[M(\bR_+)] < \infty$, then simply define $\bar \bP$ by
\be\label{e_bar_P_1}
\bar \bP(A \times B) :=\frac{
\bP[\ind_A M(B)]
}
{
\bP[M(\bR_+)]
}.
\ee
The probability measure $\bar \bP$ is usually called the
{\em Campbell measure} associated with $M$.  If $M$
is in some sense spread out evenly on its support $S$,
then we can still think of $\bar \bP$ as describing what happens when we condition
on $S$ containing a large exponential time.

\begin{example}
\label{E:competing_exponentials}
Consider the random measure $M := |\cdot \cap [0,\kappa)|$,
where $\kappa$ has an exponential distribution with rate parameter $\eta$.
By definition,
\[
\bar \bP\{\xi > x\}
= \frac{\bP[M((x,\infty))]}{\bP[M(\bR_+)]}
= \frac{e^{-\eta x} \frac{1}{\eta}}{\frac{1}{\eta}}
= e^{-\eta x},
\]
and so the distribution of $\xi$ under the Campbell measure $\bar \bP$ is the same
as the distribution of $\kappa$ under $\bP$.  According to our interpretation of
the Campbell measure, this result indicates that if $\zeta$ is a random variable
that is independent of $\kappa$ and has an exponential distribution with rate
parameter $\lambda$, then the distribution of $\zeta$ conditional on the event
$\{\zeta < \kappa\}$ should converge to the distribution of $\kappa$ as
$\lambda \downarrow 0$.  Indeed, by classical observations about
``competing exponentials'' recalled in the Introduction, the random variable $\zeta \wedge \kappa$
is independent of the event $\{\zeta < \kappa\}$
and has an exponential distribution with rate $\lambda + \eta$, so the
conditional distribution of $\zeta$ given the event $\{\zeta < \xi\}$
is exponential with rate $\lambda + \eta$ and this conditional distribution
converges to the distribution of $\kappa$ as $\lambda \downarrow 0$. 
\end{example}


\section{Markov processes and Campbell measures}\label{s:general}

In this section we introduce the assumptions used throughout the paper.
Let $((X_t), \Omega, \mathcal{F}, \Pr^x, (\theta_t),(\mathcal{F}_t))$ 
be a \textit{right process} (we sometimes denote the whole sextuple by $X$), 
see Chapter II:20 from \cite{S88},  with state space  
$E_\partial := E \cup \{\partial\}$, where $E$ is a Lusin topological space with 
Borel field $\mathcal{E}$ and $\partial$ is an adjoined cemetery state.
 Let $(P_t)_{t\geq0}$ and $(R_\lambda)_{\lambda>0}$ denote 
 the semigroup and the resolvent of $X$.

If $P_t f$  is $\mathcal{E}$-measurable whenever $f$ is a positive 
$\mathcal{E}$-measurable function and $t\geq 0$, then 
we say that $X$ is a \textit{Borel right process}.


Assume that we are in the canonical setting where $\Omega$ is the space of functions 
$\omega:\R_+\to E_\partial$ which are right continuous, and if 
$\zeta(\omega) := \inf\{t \ge 0: \omega(t) = \partial\}$, then $\omega(t) = \partial$ for
$t \ge \zeta(\omega)$.   Furthermore, $X_t(\omega):=\omega(t)$ and
$(\theta_t\omega)(s) := \omega(s+t)$.  Note that $\zeta$ is a {\em terminal time}; that is,
\begin{equation*}
\zeta = s + \zeta \circ \theta_s,
\end{equation*}
on the event $\{\zeta > s\}$ for all $s \ge 0$.
Let $\mathcal{F}_t^0$ be the natural filtration on
$\Omega$: $\mathcal{F}_t^0:=\sigma\{X_s: 0\leq s\leq t\}$.
Set $\mathcal{F}^0=\bigcup_{t}\mathcal{F}_t^0$ and for an initial law $\mu$ let $\mathcal{F}^\mu$ denote the completion of $\mathcal{F}^0$ relative to $\Pr^\mu$ and let $\mathcal{N}^\mu$ denote the $\Pr^\mu$-null sets in $\mathcal{F}^\mu$.

Set
\begin{itemize}
\item $\mathcal{F}:= \bigcap \left\{\mathcal{F}^\mu: \mu ~\text{is an initial law on}~E\right\}$.
\item $\mathcal{N} :=  \bigcap \left\{\mathcal{N}^\mu: \mu ~\text{is an initial law on}~E\right\}$.
\item $\mathcal{F}_t^\mu:= \mathcal{F}_t^0\vee \mathcal{N}^\mu$.
\item $\mathcal{F}_t := \bigcap \left\{\mathcal{F}_t^\mu: \mu~\text{is an initial law on}~E\right\}$.
\end{itemize}
The process $X$ is described by the probability family $(\Pr^x)_{x\in E}$ which satisfies
\[
\Pr^x\{X_0=x\}=1
\]
for all $x\in E$.

\begin{proposition}
\label{p:Campbell_additive_functional}
Consider a Borel right process
$X$ with state space $E$.
Suppose that the random measure $M$ on $\bR_+$ satisfies the following
conditions:
\begin{itemize}
\item
$M(\{0\}) = 0$,
\item
 $M((0,t])$ is $\cF_t$-measurable for all $t > 0$,
\item
$0 < \bP^x[M(\bR_+)] < \infty$ for all $x \in E$,
\item
$M = 0$, $\bP^\partial$-a.s.
\item
for all $s,t > 0$ and $x \in E$, $M((0,s+t]) = M((0,s]) + (M \circ \theta_s)((0,t])$ , $\bP^x$-a.s.
\end{itemize}
Then,  for any $t \ge 0$
and nonnegative $\cF_t$-measurable
random variable $Z$,
\[
\bar \bP^x[Z \ind\{\xi > t\}] = \frac{1}{\bP^x[M(\bR_+)]} \bP^x\left[Z \bP^{X_t}[M(\bR_+)]\right].
\]
\end{proposition}

\begin{proof}
By the definition of  Campbell measure, the hypotheses on $M$
 and the Markov property,
\begin{equation*}
\begin{split}
\bar \bP^x[Z \ind\{\xi > t\}]
&=
\frac{
\bP^x[Z M((t,\infty))]
}
{
\bP^x[[M(\bR_+)]
}
=
\frac{
\bP^x[Z M \circ \theta_t (\bR_+)]
}
{
\bP^x[M(\bR_+)]
}\\
&=
\frac{1}{\bP^x[M(\bR_+)]} \bP^x\left[Z \bP^{X_t}[M(\bR_+)]\right].
\end{split}
\end{equation*}
\end{proof}

\bigskip
\noindent
{\bf Proof of Theorem~\ref{t:main_1}.}  The local time at $a$ is a random measure that satisfies the hypotheses
of Proposition~\ref{p:Campbell_additive_functional}.
By the hypotheses of the theorem, $\bP^x[M(\bR_+)] = \bP^x\{T_a < \infty\} \bP^a[M(\bR_+)]$ and so
\[
\bar \bP^x[Z \ind\{\xi > t\}]
=
\frac{1}{\bP^x\{T_a<\infty\}} \bP^x\left[Z \bP^{X_t}\{T_a<\infty\}\right]
\]
for $Z$ a nonnegative $\cF_t$-measurable random variable.

Observe that $\bP^x\{T_a<\infty\} = \bP^x\{0 < K_a < \infty\}$.  The random time $K_a$ is co-optional
and it follows from the remark after equation (62.24) of \cite{S88} that the distribution of 
$(X_t)_{0 \le t < \xi}$ under the Campbell measure $\bar \bP^x$
is the same as the distribution of $(X_t)_{0 \le t < K_a}$ under $\bP^x$ conditional
on $\{T_a < \infty\}$. \qed

\section{Excessive functions and Doob $h$-transforms}\label{s:doob_general}

Recall that a function $h :E\rightarrow \R_+\cup\{+\infty\}$ is \textit{excessive} if
the following two conditions are satisfied:
\begin{itemize}
\item[(1)]
\[
\Pr^x [h(X_t)] \leq h(x)
\]
for all $t\geq 0$ and $x\in E$.
\item[(2)]
\[
\lim_{t\downarrow 0} \Pr^x [h(X_t)] = h(x)
\]
for all $x\in E$.
\end{itemize}

\begin{remark}
Suppose that $M$ satisfies hypotheses of Proposition~\ref{p:Campbell_additive_functional}.  
Set $h(x) = \Pr^x [M(\bR_+)]$.
Observe that
$\Pr^x [h(X_t)]
=
\Pr^x[M \circ \theta_t(\bR_+)]
=
\Pr^x[M((t,\infty))]$
and it is clear that $h$ is excessive.
\end{remark}
\begin{example}
The function
\[
x\mapsto \Pr^x\{T_a<\infty\}
\]
is excessive.
\end{example}

The following result is well-known at various levels of generality.

\begin{theorem}\label{t:conditioning}
Let $((X_t), \Omega, \mathcal{F}, \Pr^x, (\theta_t),(\mathcal{F}_t))$ be a Borel right process on a Lusin space $E$ and let $(P_t)_{t\geq 0}$ be its Borel semigroup. Suppose $h:E\rightarrow \R_+$ is a positive Borel excessive function.
The operators $(P^h_t)_{t \geq 0}$ defined by
\[
P_t^h g(x) = \frac{1}{h(x)} P_t g h (x), \quad x\in E_h:= \{x\in E: 0<h(x)<\infty\}
\]
comprise a submarkovian semigroup that corresponds to a Borel right process with state space $E_\partial:= E_h\cup\{\partial\}$.
\end{theorem}
\begin{proof}
By Theorem 62.19 from \cite{S88} (see also (62.23) in \cite{S88}) we know that $(P_t^h)_{t\geq 0}$ defines the semigroup of a right process on $E_\partial:= E_h\cup\{\partial\}$. It
is clear that this semigroup is Borel.
\end{proof}

\begin{remark}\label{r:doob}
The Markov process with the semigroup $(P_t^h)_{t \ge 0}$ of Theorem~\ref{t:conditioning} is called the Doob $h$-transform of the original Markov process
(with respect to the excessive function $h$).
If $a \in E$ is such that for all $x \in E$, $\bP^x\{T_a<\infty\} > 0$ and
$\bP^a\{T_a = 0\} = 1$ where $\kappa$
is an independent exponential random variable with rate parameter $\lambda$, then
we see from Theorem \ref{t:main_1} that the distribution under $\bP^x$ of $(X_t)_{0 \le t < \kappa}$ conditional on the event $\{X_\kappa = a\}$
converges as $\lambda \downarrow 0$
to the distribution under $\bQ^x$ of $(X_t)_{0 \le t < \zeta}$, where $\bQ^x$ is now the Doob $h$-transform distribution corresponding
to the excessive function $x \mapsto \bP^x[M(\bR_+)]$, where $M$ is the local time at $a$ or, equivalently, to the excessive function
$x \mapsto \bP^x\{T_a < \infty\}$.
\end{remark}

\section{Bang-bang processes and excursions}\label{s:bangbang}

\subsection{Brownian motion with negative drift}\label{s:bang-bang BM}

Suppose that $X$ is a Brownian motion with negative drift $-\mu$, $\mu>0$, and $a=0$ in the context of Theorem~\ref{t:main_1}.
Let $X^h$ be the Doob $h$-transform process corresponding to the excessive function
$x \mapsto \bP^x\{T_0 < \infty\}$.
Recall from Theorem~\ref{t:main_1} that the behavior of the process $X^h$ started at $0$ is what we see
if we start the process $X$ at $0$ and then kill it at the start of the first
infinite excursion away from $0$. We would like to show that this is the same as taking the bang-bang
Brownian motion that evolves as Brownian motion with drift $-\mu$ when it is positive
and as Brownian motion with drift $+\mu$ when it is negative, and killing that
bang-bang Brownian motion when the local time at $0$ exceeds an independent
exponential random variable with rate parameter $\mu$.

Consider excursions from the point $0$.
Formula (50.3) in Section VI.50 of \cite{RW00} gives that
\begin{equation}\label{e:entrance_law}
\int_0^\infty  e^{-\lambda t} n_t(x) dt = \frac{r_\lambda(0,x)}{\bP^a \int_0^\infty e^{-\lambda s} \,dL^0_s},
\end{equation}
where $n_t(x) dx$ is the entrance ``law" for the It\^o excursion measure and $\lambda R_\lambda(x,\cdot)$ is the $\Pr^x$ law of $X_T$ where $T$ is an independent exponential random variable with rate $\lambda$.
Note that
\begin{equation}\label{e:L_a}
\Pr^0 \int_0^\infty e^{-\lambda s} \,dL^0_s = r_\lambda(0,0) = \int_0^\infty e^{-\lambda s}p_s(0,0)\,ds
\end{equation}
where $r_\lambda(x,y)$ and $p_t(x,y)$ are the resolvent and transition densities of $X$ with respect to Lebesgue measure.
From \eqref{e:entrance_law}, \eqref{e:L_a} and
\begin{equation}\label{e:density}
\Pr^x\{W_T-\mu T\in dz\} = \frac{\lambda}{\sqrt{2\lambda+\mu^2}} e^{-\mu(z-x)-|z-x|\sqrt{2\lambda+\mu^2}}dz.
\end{equation}
for $W$ a Brownian motion and $T$ an exponential random variable with rate parameter $\lambda$,
\begin{equation}
\int_0^\infty  e^{-\lambda t} n_t(x) \, dt = \exp(- \mu x - |x| \sqrt{2 \lambda + \mu^2}).
\end{equation}
The positive excursions are all finite.  The probability that a Brownian motion with drift  $-\mu$ ever hits $0$ started from $x < 0$ is $ \exp(2 \mu x)$, and so the entrance law
$n_t^f(x) dx$ for the It\^o excursion measure on negative excursions of finite length satisfies
\begin{eqnarray*}
\int_0^\infty  e^{-\lambda t} n_t^f(x) dt &=&  \exp(- \mu x - |x| \sqrt{2 \lambda + \mu^2}) \exp(2 \mu x)\\
&=&  \exp(\mu x - |x| \sqrt{2 \lambda + \mu^2}).
\end{eqnarray*}
This is
\begin{equation*}
\int_0^\infty e^{-\lambda t} m_t(x) \, dt
\end{equation*}
where $m_t(x) dx$ is the entrance law for the It\^o excursion  measure on negative excursions 
for Brownian motion with drift $+ \mu$.

The rate at which infinite excursions come along in local time can be found by seeing that
\[
\begin{split}
& \int_{-\infty}^0 e^{- \mu x - |x| \sqrt{2 \lambda + \mu^2}} -  e^{\mu x - |x| \sqrt{2 \lambda + \mu^2}}\, dx \\
& \quad  = [(\sqrt{2 \lambda + \mu^2} - \mu]^{-1} - [(\sqrt{2 \lambda + \mu^2} + \mu]^{-1}\\
& \quad  = \mu/\lambda, \\
\end{split}
\]
and so the rate is $\mu$.

By the discussion around (50.7) in \cite{RW00}, if we have Brownian motion with drift $- \mu$, we start it below zero and we condition it to hit zero, then up to the time it hits zero we see a Brownian motion with drift $+\mu$.

The process $(Y_t)_{t \ge 0}$ defined via the SDE
\[
dY_t = dU_t - \mu \sgn{(Y_t)} \, dt
\]
for $\mu\in\R$ and $U_t$ a standard Brownian motion is called bang-bang Brownian motion or Brownian motion with alternating drift -- see \cite{GS00} and Appendix 1.15 in \cite{Bor02}.

Putting the above together it appears that $X^h$ started at $0$ is indeed a bang-bang Brownian motion
killed at $0$ according to local time with rate $\mu$.  There is, however, a missing ingredient in this identification.
We have not identified the process obtained by concatenating together in the usual way
the points in a Poisson process of positive and finite length negative excursions of Brownian motion with
drift $-\mu$ with a bang-bang Brownian motion.  We will take a slightly different route in the
remainder of this section to establish that $X^h$ is bang-bang Brownian motion suitably killed at $0$.

\subsection{Excursions of a Markov process from a regular point}

We briefly review some of the concepts from It\^o excursion theory that we need. 
We follow \cite{RW00} VI 42-50 and remark that the results there hold in our setting 
(see also \cite{Sal86a, Sal86b, ito71}). 

Suppose $X$ is a Borel right process with Lusin state space $E$. We assume the point $a\in E$ is a regular point, that is
\begin{equation*}
\Pr^a\{T_a=0\}=1
\end{equation*}
where
\begin{eqnarray*}
\mathcal{M} &=& \{t\geq 0: X_t=a\}\\
T_a &=& \inf\{t>0: t\in \mathcal{M}\}.
\end{eqnarray*}
One can then show that the function $\psi(x) :=\bP^x\left[e^{-T_a}\right]$ is the $1$-potential of some PCHAF (perfect, continuous, homogeneous, additive functional) $L$
\begin{eqnarray*}
\psi(x) = \bP^x \left[\int_0^\infty e^{-s} \, dL_s\right]
\end{eqnarray*}
for every $x\in E$. The additive functional
 $L$ is the local time of $X$ at $a$ and the set of points of increase of $L$ 
is exactly the closed random set $\mathcal{M}$.
\begin{remark}
Any PCHAF which grows only on $\mathcal{M}$ must be a multiple of $L$. 
\end{remark}
The process $\gamma_t:=\inf\{u: L_u>t\}$, where $\inf \emptyset = +\infty$, is a killed subordinator under $\Pr^a$
that is sent to $+\infty$ at its death time. An \textit{excursion} is a right continuous function $f:\R_+ \rightarrow E$.
such that if
\begin{equation*}
T_a(f)=\inf\{t>0: f(t)=a\},
\end{equation*}
then $f(t) = a$ for $t > T_a(f)$. Let $U$ denote the set of all excursions.
\begin{definition}
The point process of excursions from $a$ is
\begin{equation*}
\Pi := \{(t,e_t):\gamma_t\neq \gamma_{t-}\}
\end{equation*}
where $e_t\in U$,  the excursion at local time $t$, is
\[
e_t(s) =
\begin{cases} X_{\gamma_{t-}+s}, &\mbox{if } 0\leq s<\gamma_t-\gamma_{t-},\\
  a, & \mbox{otherwise. }
\end{cases}
\]
We can also think of $\Pi$ as a $\mathbb{Z}_+\cup \{\infty\}$-valued random measure. For any Borel set $A\subset \R_{++}\times U$
\begin{equation*}
N(A) := \#(A\cap \Pi).
\end{equation*}
Denote by $U_\infty:=\{f\in U: T_a(f)=\infty\}$ the infinite excursions and by $U_0:=U\setminus U_\infty$ the finite excursions.
\end{definition}

The main result of excursion theory says that there exists a $\sigma$-finite measure $n$ on $U$ such that $n(U_\infty)<\infty$, if $N'$ is a Poisson random measure on $\R_{++}\times U$ with expectation measure
 $\mathrm{Leb}\otimes n$, 
\begin{equation*}
\zeta := \inf\{t>0: N((0,t]\times U_\infty)>0\},
\end{equation*}
and
\begin{equation*}
\zeta' :=\inf\{t>0: N'((0,t]\times U_\infty)>0\},
\end{equation*}
then  the random measure $N = N(\cdot \cap (0,\zeta] \times U)$ under $\Pr^a$ and 
the random measure $N'(\cdot\cap (0,\zeta']\times U)$ have the same distribution.

\subsection{Construction of the bang-bang process}\label{s:bangbang_2}

In this section we construct a general version of the bang-bang process and, as a result, prove Theorem \ref{t:main_2}.
Assume throughout that the process $X$ satisfies the conditions of Theorem \ref{t:main_2}.

Let the process $X^h$ be the $h$-transform of $X$ using
\begin{equation}\label{e:h_density}
h(x) := \Pr^x\{T_a<\infty\} = \frac{r(x,a)}{ r(a,a)},
\end{equation}
where $r$ is the density for the 0-resolvent of $X$ (i.e. $r(x,y) = \int_0^\infty p_t(x,y) \, dt$).
By Theorem \ref{t:conditioning} $X^h$ is a Borel right process. 


We construct a new process from $X^h$ as follows.  We run $X^h$ until it dies, then we start another copy of $X^h$ from $a$, wait until it dies, and so on. Call this process $X^b$.  This is a special case of the construction of a resurrected process in \cite{Fitz91,meyer75}. By \cite{meyer75} we get that $X^b$ is a Borel right process.
Let
\begin{equation}\label{e_Resolvent^h}
R_\lambda^h g(x) = h(x)^{-1} R_\lambda (gh) (x)
\end{equation}
be the resolvent of the $h$-transform of $X$.
Note that $(R_\lambda^h)_{\lambda>0}$ satisfies the resolvent equation
\begin{equation}\label{e:R^h_property}
R_\lambda^h - R_\chi^h + (\lambda - \chi) R_\lambda^h R_\chi^h = 0, \quad  \lambda, \chi>0.
\end{equation}
The density $r_\lambda^h(x,a)$ of $R_\lambda^h$ may be treated informally as
\begin{equation}\label{e:resolvent_delta}
R_\lambda^h \delta_a(x),
\end{equation}
where $\delta_a$ is the ``Dirac delta function at $a$", and such manipulations can be made rigorous
using suitable approximations.

If $T$ is an independent exponential time with rate $\lambda$ and $\zeta$ is the time that $X^h$ dies, then
\begin{eqnarray*}
\bP^x[f(X_T^b)] &=& \bP^x[f(X_T^h), T < \zeta] + \bP^x[f(X_T^b), \zeta \le T]\\
&=&\bP^x[f(X_T^h), T < \zeta] + \bP^x[\exp(-\lambda \zeta)]  \bP^a[f(X_T^b)].
\end{eqnarray*}
Now,
\begin{equation*}
\bP^x\left[\int_0^\zeta \exp(-\lambda t) \, dt\right] = \frac{1}{\lambda} (1 - \bP^x[\exp(-\lambda \zeta)] )
\end{equation*}
and
\begin{equation*}
\bP^x\left[\int_0^\zeta \exp(-\lambda t) \, dt\right ] = R_\lambda^h 1(x),
\end{equation*}
so,
\begin{equation*}
R_\lambda^b f(x) = R_\lambda^h f(x) + (1 - \lambda R_\lambda^h 1(x)) R_\lambda^b f(a)
\end{equation*}
for all $x$. In particular, we can put in $x=a$ and solve to find that
\begin{equation*}
R_\lambda^b f(a) = R_\lambda^h f(a) / (\lambda R_\lambda^h 1(a))
\end{equation*}
and hence
\begin{equation}\label{e:resolve_second_construction}
R_\lambda^b f(x) = R_\lambda^h f(x) + (1 - \lambda R_\lambda^h 1(x)) R_\lambda^h f(a) / (\lambda R_\lambda^h 1(a)).
\end{equation}

Use $h(x) = \frac{R_0 \delta_a (x) } {r_0(a,a)}$ and the resolvent equation to get
\begin{eqnarray*}
\lambda R_\lambda h(a) &=& \frac{\lambda R_\lambda R_0 \delta_a (a)}{ r_0(a,a)} \\
&=& \frac{(R_0 - R_\lambda) \delta_a (a)}{ r_0(a,a)} \\
&=& \frac{r_0(a,a) - r_\lambda(a,a)} { r_0(a,a)} \\
&=& 1 - \frac{r_\lambda(a,a)}{r_0(a,a)}.
\end{eqnarray*}
This transforms \eqref{e:resolve_second_construction} into
\begin{equation}\label{e:resolve_second_construction2}
R_\lambda^b f(x) = R_\lambda^h f(x) + (1 - \lambda R_\lambda^h 1(x)) R_\lambda^h f(a) \frac{r_0(a,a)}{r_0(a,a)-r_\lambda(a,a)}.
\end{equation}

\begin{remark}
If $X$ is continuous, then $X^b$ is also continuous.
\end{remark}

\begin{remark}
Note that $X^b$ has resolvent densities with respect to the measure $m$. 
\end{remark}

\bigskip
\noindent
{\bf Proof of Theorem \ref{t:main_2}.}
From Theorem \ref{t:main_1} $(X_t^h)_{0\leq t<\zeta}$ under $\Pr^x$ is distributed as 
$(X_t)_{0\leq t<K_a}$ under $\Pr^x$ conditioned on $\{T_a<\infty\}$. 
The process $(X^b_t)_{t\geq 0}$ under $\Pr^x$ comes from pasting together $(X_t^h)_{0\leq t<\zeta}$ 
under $\Pr^x$ with independent identically distributed copies of $(X_t^h)_{0\leq t<\zeta}$ under 
$\Pr^a$. As a result, the process  $(X^b_t)_{t\geq 0}$ under $\Pr^x$ can be equivalently constructed
 by pasting together $(X_t)_{0\leq t<K_a}$ under $\Pr^x$ conditioned on $\{T_a<\infty\}$ 
with independent identically distributed copies of $(X_t)_{0\leq t<K_a}$ under $\Pr^a$.

Let $L$ be the local time of $X$ at $a$, $M$ the local time of $X^b$ at $a$ and $K$ be the time that the first copy of $(X_t)_{0\leq t<K_a}$ is killed. We see by the above that $M_K$ under $\Pr^x$ has the same distribution as $L_{K_a}$ under $\Pr^a$. The proof of Theorem \ref{t:main_2} is concluded by noting that $L_{K_a}$ is an exponential with rate $n(U_\infty)$. \qed

\section{Doob $h$-transforms for one-dimensional diffusions: characteristics}\label{s:doob}

We follow \cite{Bor02} and \cite{IM96} in defining a general one-dimensional diffusion and its characteristics.

Let $I=(\ell,r)$ with $-\infty\leq \ell<r\leq \infty$
and suppose that 
$((X_t), \Omega, \mathcal{F}, \Pr^x, (\theta_t),(\mathcal{F}_t))$ is a Borel right process, see Section \ref{s:general}, taking values in $I\cup \{\partial\}$. $X$ is called a \textit{linear} (or \textit{one-dimensional}) \textit{diffusion} if for all $x\in I$,
\[
\Pr^x\{\omega: t\mapsto X_t(\omega)~\text{is continuous on}~[0,\zeta)\}=1,
\]
where $\zeta$ is the lifetime of $X$.

We only consider \textit{regular diffusions}; that is, diffusions such that for all $x,y\in I$
\[
\Pr^x\{T_y<\infty\}>0,
\]
where $T_y:=\inf\{t:X_t=y\}$ -- any state $y$ can be reached in finite time with positive probability from any state $x$.

The diffusion $X$ determines  three basic Borel measures on the state space $I$:
a {\em scale measure} $s$, a {\em speed measure} $m$,  and a {\em killing measure} $k$ (see \cite{IM96}).
It turns out to be convenient not to specify these objects absolutely but only up to a constant.
If $(s^*, m^*, k^*)$ and $(s^{**}, m^{**}, k^{**})$ are two triples of these objects, then
$s^{**} = c s^*$ for some strictly positive constant $c$, in which case $m^{**} = c^{-1} m^*$
and $k^{**} = c^{-1} k^*$.  The scale measure $s$ is  diffuse.  Both the scale measure and the speed measure
have full support and assign
finite mass to intervals of the form $(y,z)$, where $\ell < y < z < r$.
If $(P_t)_{t \ge 0}$ is the transition semigroup of $X$, then there exists a density $p$ that is
strictly positive, jointly continuous in all variables, and symmetric such that
\[
P_t(x,A)=\int_A p(t;x,y) \, m(dy), \quad  \text{$x\in I$, $t>0$, and $A\in \mathcal{B}(I)$},
\]
where $\mathcal{B}(I)$ are the Borel subsets of $I$.
The killing measure $k$ assigns  finite mass to intervals of the form $(y,z)$, where $\ell < y < z < r$
and
\[
\Pr^x\{X_{\zeta-}\in A, \,  \zeta<t\} = \int_0^t \int_A  p(s;x,y)  \, k(dy) \, ds , \quad  A\in\mathcal{B}(I).
\]

We outline the recipes from \cite{IM96} for defining measures $s_{ab}, m_{ab}, k_{ab}$
on an interval $(a,b)$, $\ell < a < b < r$, such that if $s,m,k$ are the scale, speed and killing measures for $X$,
then there is a strictly positive constant $c_{ab}$ depending on $a,b$ such that $s(B) = c_{ab} s_{ab}(B)$,
$m(B) = c_{ab}^{-1} m_{ab}(B)$, and $k(B) = c_{ab}^{-1} k_{ab}(B)$ for $B \subseteq (a,b)$.
For $x\in (a,b)$,
define the \textit{hitting probabilities}
\[
p_{ab}(x) := \Pr^x\{T_a<T_b\},
\]
and
\[
p_{ba}(x) := \Pr^x\{T_b<T_a\},
\]
and the mean \textit{exit time}
\[
e_{ab}(x) := \Pr^x[T_a\wedge T_b \wedge \zeta].
\]
For ease of notation, we drop the subscripts for the
moment and write $s,m,k$ instead of $s_{ab}, m_{ab}, k_{ab}$.
Then
\be\label{e:scale_diffusion}
s(dx) :=p_{ab}(x)p_{ba}(dx)-p_{ba}(x)p_{ab}(dx)
\ee
\be\label{e:killing_diffusion}
k(dx) :=\frac{D_s^+p_{ab}(dx)}{p_{ab}(x)} = \frac{D_s^+p_{ba}(dx)}{p_{ba}(x)}
\ee
\be\label{e:speed_diffusion}
m(dx) := -[D_s^+e_{ab}(dx)-e_{ab}(x)k_{ab}(dx)]
\ee
for  $x\in (a,b)$, where
\[
D^+_{s}f(x) = \lim_{\eta\downarrow x}\frac{f(\eta)-f(x)}{s(\eta)-s(x)},
\]
and
\[
D^-_{s}f(x) = \lim_{\eta\uparrow x}\frac{f(\eta)-f(x)}{s(\eta)-s(x)}
\]
for a function $f:(a,b) \to \bR$ and, with a standard abuse of notation, as well as using $s$
to denote the scale measure we write $s$ for any {\em scale function} such that
\[
s(z) - s(y) =\int_y^z \, s(dx).
\]
For $\alpha>0$ the \textit{Green function} $r_\alpha(x,y)$ is given by
\begin{equation*}
r_\alpha(x,y):= \int_0^\infty e^{-\alpha t} p(t;x,y)\,dt,
 \end{equation*}
where, as above, $p(t;x,y)$ is the transition density with respect to the speed measure $m$.
Put
\[
r_0(x,y):=\lim_{\alpha\downarrow 0} r_\alpha(x,y).
\]
The diffusion $X$ is said to be \textit{recurrent} if
\[
\Pr^x\{T_y<\infty\}=1
\]
for all $x,y\in I$. A diffusion that is not recurrent is said to be \textit{transient}.
The diffusion $X$ is transient if and only if for all $x,y\in I$
\[
r_0(x,y)<\infty.
\]

\begin{remark}
If the killing measure is null ($k\equiv 0$), then
\[
r_0(x,y)
= \int_0^\infty p(t;x,y)\,dt
=\lim_{a\downarrow \ell, b\uparrow r}\frac{(s(x)-s(a))(s(b)-s(y))}{s(b)-s(a)}, \quad  x\leq y.
\]
\end{remark}

For a regular diffusion $X$ there exists (see \cite{IM96})
a family of random variables $\{L(t,x): x\in I, \, t\geq 0\}$ (sometimes also denoted by $L_t^x$) called the \textit{local time} of $X$, such that
\begin{itemize}
\item [I.]
\[
\int_0^t \ind_A (X_s)\,ds = \int_A L(t,x) \, m(dx), \; \Pr^x-\text{a.s.}, \; A\in \mathcal{B}(I),
\]
\item [II.]
\be\label{e:local_time_diffusion}
L(t,x)=\lim_{\epsilon \downarrow 0} \frac{\int_0^t\ind_{(x-\epsilon,x+\epsilon)}(X_s)\,ds}{m((x-\epsilon,x+\epsilon))},  \; \Pr^x-\text{a.s.}~
\ee
\item [III.] For any $s<t$,
\[
L(t,x,\omega)= L(s,x,\omega)+ L(t-s,x,\theta_s(\omega)), \;\Pr^x-\text{a.s.}
\]
\end{itemize}
One has
\[
\bP^x\left[\int_0^\infty e^{-\alpha t} \, dL(t,y)\right] = r_\alpha(x,y).
\]

For a fixed $x$ the process $L^x:=(L(t,x))_{t\geq 0}$ is called the \textit{local time process} of $X$ at the point $x$.

Suppose that $X$ is a regular, transient diffusion with null killing measure and $h:I\to \R_+$ is a strictly positive excessive function.
Since two strictly positive excessive functions that are multiples of each other lead to the same Doob $h$-transform, we may assume
for some $x_0\in I$ that $h(x_0)=1$.
For $\lambda>0$ and for some fixed reference point $a\in I$ define the functions $\psi_\lambda$ and $\phi_\lambda$  by
\begin{equation}\label{e:psi_+}
\psi_\lambda(x) =
  \begin{cases}
       \bP^x[\exp(-\lambda T_a)],    & \quad x\leq a, x\in\text{int}(I), \\
       1/\bP^a[\exp(-\lambda T_x)],  & \quad x\geq a, x\in\text{int}(I), \\
  \end{cases}
\end{equation}
and
\begin{equation}\label{e:psi_-}
\phi_\lambda(x) =
  \begin{cases}
      \bP^x[\exp(-\lambda T_a)],    & \quad x\geq a, x\in\text{int}(I), \\
      1/\bP^a[\exp(-\lambda T_x)],  & \quad x\leq a, x\in\text{int}(I). \\
  \end{cases}
\end{equation}
Note that
\[
\lim_{\lambda\downarrow 0} \bP^x\left[e^{-\lambda T_a}\right] = \Pr^x\{T_a<\infty\}.
\]
As a result the functions $\psi_0:=\lim_{\lambda \downarrow 0}\psi_\lambda$ and $\phi_0:=\lim_{\lambda \downarrow 0}\phi_\lambda$ satisfy
\begin{equation}\label{e:psi_0_+}
\psi_0(x) =
  \begin{cases}
         \Pr^x\{T_a<\infty\},   & \quad  x\leq a, x\in\text{int}(I), \\
         1/\Pr^a\{T_x<\infty\}, & \quad  x\geq a, x\in\text{int}(I), \\
  \end{cases}
\end{equation}
and
\begin{equation}\label{e:psi_0_-}
\phi_0(x) =
  \begin{cases}
         \Pr^x\{T_a<\infty\},    & \quad  x\geq a, x\in\text{int}(I), \\
         1/\Pr^a\{T_x<\infty\}, & \quad  x\leq a, x\in\text{int}(I). \\
  \end{cases}
\end{equation}
There is also the following relationship between the Green function $r_\alpha(x,y)$ and the functions $\psi_\lambda, \phi_\lambda$
\be
r_\lambda(x,y) =
\begin{cases}
w_\lambda^{-1}\psi_\lambda(x)\phi_\lambda(y),& x\leq y, \\
w_\lambda^{-1}\psi_\lambda(y)\phi_\lambda(x),& x\geq y,
\end{cases}
\ee
where the \textit{Wronskian}
\[
w_\lambda:=D_s^+\psi_\lambda(x)\phi_\lambda(x)-\psi_\lambda(x)D_s^+\phi_\lambda(x)
\]
is independent of $x$.

By \cite[II.5.30]{Bor02},  there is
a probability measure $\nu$ called the \textit{representing measure} of $h$ such that
\begin{equation}\label{e_aexc}
h(x)
=
\int_{(\ell,r)} \frac{r_0(x,y)}{r_0(x_0,y)}\,\nu(dy) +
\frac{\phi_0(x)}{\phi_0(x_0)} \, \nu(\{\ell\}) +
\frac{\psi_0(x)}{\psi_0(x_0)} \, \nu(\{r\}).
\end{equation}

Note that
\[
\lim_{y \to \ell} \bP^x\{T_y < \infty\}
= \lim_{y \to \ell} \frac{r_0(x,y)}{r_0(y,y)}
= \lim_{y \to \ell} \frac{\psi_0(y)\phi_0(x)}{\psi_0(y)\phi_0(y)}
= \frac{\phi_0(x)}{\phi_0(\ell+)}.
\]
Similarly,
\[
\lim_{y \to r} \bP^x\{T_y < \infty\}
= \frac{\psi_0(x)}{\psi_0(r-)}.
\]
Thus,
\begin{equation}
\label{e_aexc_interpret}
\begin{split}
h(x)
& :=
\bP^x\biggl[\int_{(\ell,r)} L_\infty^y / r_0(x_0,y) \, \nu(dy) \\
& \quad  +
\ind\left\{\lim_{t \to \infty} X_t = \ell\right\} \frac{\phi_0(\ell+) }{ \phi_0(x_0)} \, \nu(\{\ell\}) \\
& \quad 
+ \ind\left\{\lim_{t \to \infty} X_t = r\right\} \frac{\psi_0(r-)}{\psi_0(x_0)} \, \nu(\{r\}) \biggr]. \\
\end{split}
\end{equation}

\begin{theorem}
\label{t:h_diffusion_chars}
Let $X$ be a regular, transient diffusion with null killing measure,
speed measure $m$ and scale function $s$.
Suppose that $h$ is a strictly positive excessive function such that $h(x_0) = 1$
and $h$ has representing measure $\nu$.  The Doob $h$-transform is
a regular diffusion with the following characteristics:
 \begin{itemize}
 \item Scale measure
  \begin{equation}\label{e:h_scale}
 s^h(dy) = h^{-2}(y) \, s(dy).
\end{equation}
\item Speed measure
 \begin{equation}\label{e:h_speed}
 m^h(dy) = h^2(y) \, m(dy).
 \end{equation}
 \item Killing measure
  \begin{equation}\label{e:h_killing}
 k^h(dy)= \frac{h(x_0)h(y)}{r_0(x_0,y)} \, \nu(dy).
\end{equation}
\end{itemize}
\end{theorem}

\begin{proof}
Define the random measure $\bar M$ on $\bR_+ \cup \{+\infty\}$ by
\[
\bar M(B) := \int_{(\ell,r)} \int_B dL_t^y / r_0(x_0,y) \, \nu(dy), \quad  B \subseteq \bR_+,
\]
and
\[
\begin{split}
\bar M(\{+\infty\})
& :=
\ind\left\{\lim_{t \to \infty} X_t = \ell\right\} \frac{\phi_0(\ell+)} {\phi_0(x_0)} \, \nu(\{\ell\}) \\
& \quad  +
\ind\left\{\lim_{t \to \infty} X_t = r\right\} \frac{\psi_0(r-) }{ \psi_0(x_0)} \, \nu(\{r\}). \\
\end{split}
\]

With a small change in the meaning of the notation used previously for a Campbell measure,
define the probability measure $\bar \bP^x$ on $\Omega \times (\bR_+ \cup \{+\infty\})$ by
\[
\bar \bP^x\{A \times B\}
=
\frac{1}{h(x)}\bP^x\left[\ind_A \bar M(B)\right].
\]
for $B \subseteq \bR_+ \cup \{+\infty\}$.

Writing $\tilde \bP^x$ for the distributions of the $h$-transformed process, we have for any
finite stopping time $R$ and nonnegative
$\cF_R$-measurable random variable $Z$ that
\[
\tilde \bP^x[Z \ind\{\zeta > R\}]
=
\bar \bP^x[Z \ind\{\xi > R\}]
=
\frac{1}{h(x)} \bP^x[Z h(X_R)].
\]
In particular, the distribution of $\zeta$ under $\tilde \bP^x$ is that of $\xi$ under $\bar \bP^x$.

Recall that $p_{ab}(x) := \Pr^x\{T_a<T_b\}$ and  $p_{ba}(x) := \Pr^x\{T_b<T_a\}$.
Put $p^h_{ab}(x) := \tilde \bP^x\{T_a<T_b\}$ and $p^h_{ba}(x) := \tilde \bP^x\{T_b<T_a\}$.
Setting $T := T_a \wedge T_b$, we have
\begin{equation}
\begin{split}
p^h_{ab}(x)
&=
\int_{\Omega \times (\bR_+ \cup \{+\infty\})} \ind\{T_a(\omega) < T_b(\omega), \, T_a(\omega) < u\} \, \bar \bP^x(d\omega, du) \\
&=
\int_{\Omega \times (\bR_+ \cup \{+\infty\})} \ind\{X_T(\omega) = a, \,  T(\omega) < u\} \, \bar \bP^x(d\omega, du) \\
&=
\frac{1}{h(x)}\Pr^x\left[\ind\{X_T = a\} h(X_T)\right] \\
&= h(a) \bP^x\{T_a<T_b\}/h(x). \\
\end{split}
\end{equation}
Thus,
\[
p^h_{ab}(x) = h(a) p_{ab}(x) / h(x)
\]
and, by a similar argument,
\[
p^h_{ba}(x) = h(b) p_{ba}(x) / h(x).
\]
Put
$e_{ab}^h(x) := \tilde \bP^x[T_a \wedge T_b \wedge \zeta]$.
Then
\begin{equation*}
\begin{split}
e_{ab}^h(x)
&= \bar \bP^x[\xi \, \ind\{\xi \le T_a \wedge T_b\}] + \bar \bP^x[T_a \wedge T_b \, \ind\{T_a \wedge T_b < \xi\}] \\
&=\frac{1}{h(x)}\bP^x\left[\int_0^{T_a \wedge T_b} t \, \bar M(dt)\right]
+ \frac{1}{h(x)}\bP^x[T_a \wedge T_b \, h(X_{T_a \wedge T_b})].\\
\end{split}
\end{equation*}
Now,
\[
\begin{split}
\bP^x\left[\int_0^{T_a \wedge T_b} t \, \bar M(dt)\right]
& =
\bP^x\left[\int_0^\infty M([t,+\infty] \cap [0, T_a \wedge T_b))] \, dt \right]\\
& =
\bP^x\left[\int_0^\infty \ind\{t < T_a \wedge T_b\}
   \, \bP^{X_t} [\bar M([0, T_a \wedge T_b))]] \, dt \right] \\
& =
\bP^x\left[\int_0^\infty \ind\{t < T_a \wedge T_b\}
\, \bP^{X_t}\left[\int L_{T_a \wedge T_b}^y / r_0(x_0,y) \, \nu(dy)\right] \, dt\right] \\
& =
\int_a^b G_{ab}(x,z) \int_a^b G_{ab}(z,y) / r_0(x_0,y) \, \nu(dy) \, m(dz),\\
\end{split}
\]
where
\begin{equation}
\label{e:def_G_ab}
G_{a,b}(x,y)=G_{a,b}(y,x) := \frac{(s(x)-s(a))(s(b)-s(y))}{s(b)-s(a)}, \quad  a < x\leq y < b.
\end{equation}
Also, by \cite[Equation~4.1]{MR0375477},
\[
\begin{split}
\bP^x[T_a \wedge T_b \, h(X_{T_a \wedge T_b})]
& =
\bP^x[T_a \, \ind\{T_a < T_b\} h(a)] \\
& \quad  +
\bP^x[T_b \, \ind\{T_b < T_a\} h(b)] \\
& =
\int_a^b G_{a,b}(x,y) \frac{s(b)-s(y)}{s(b)-s(a)} \, m(dy) \\
& \quad  +
\int_a^b G_{a,b}(x,y) \frac{s(y)-s(a)}{s(b)-s(a)} \, m(dy).\\
\end{split}
\]
Thus
\begin{equation}
\label{e_ab^h}
\begin{split}
e^h_{ab}(x)
&=
\frac{1}{h(x)}\int_a^b G_{ab}(x,z) \int_a^b \frac{ G_{ab}(z,y)}{ r_0(x_0,y)} \, \nu(dy) \, m(dz) \\
&\quad  + \frac{1}{h(x)} \biggl[h(a) \int_a^b G_{a,b}(x,y) \frac{s(b)-s(y)}{s(b)-s(a)} \, m(dy) \\
&\quad  + h(b) \int_a^b G_{a,b}(x,y) \frac{s(y)-s(a)}{s(b)-s(a)} \, m(dy) \biggr]. \\
\end{split}
\end{equation}

From \eqref{e:scale_diffusion}, $s_{ab}^h$ (which we write as $s^h$ for ease of notation), is given by
\begin{equation}
\begin{split}
s^h(dx) &= p^h_{ab}(x) p^h_{ba}(dx) - p^h_{ba}(x) p^h_{ab}(dx)\\
&= h(a) \frac{p_{ab}(x)}  {h(x)}  h(b)  \left(\frac{p_{ba}(dx)}{ h(x)}  - \frac{p_{ba}(x) h(dx)} { h^2(x)}\right) \\
&-h(b) \frac{p_{ba}(x)}{ h(x)}   h(a)\left( \frac{p_{ab}(dx)}{ h(x)}  - \frac{p_{ab}(x) h(dx)}{ h^2(x)}\right)\\
&= h(a) h(b) h^{-2}(x) \, s(dx).\\
\end{split}
\end{equation}
Note that this agrees with \eqref{e:h_scale} apart from the constant multiple $h(a) h(b)$.

We next turn to \eqref{e:killing_diffusion} to determine $k_{ab}^h$,
which write as $k^h$.
By the quotient rule,
\begin{equation*}
\begin{split}
D^+_{s^h} p^h_{ab}(x) &=h(a) \frac{h^2(x)} {h(a) h(b)}
\left[-\frac{- p_{ab}(x) D_s^+ h(x)}{ h^2(x)}   + \frac{D_s^+ p_{ab}(x)}{h(x)}\right]
\\
&= \frac{1} { h(b)}
\left[-p_{ab}(x) D_s^+ h(x)   + D_s^+p_{ab}(x)h(x)\right],
\end{split}
\end{equation*}
where we stress that the derivatives are with respect to the original scale measure $s = s_{ab}$
rather than $s^h = s_{ab}^h$.

We now have to determine the measure
\[
D^+_{s^h}p^h_{ab}(dx).
\]
Because the original process $X$ doesn't have any killing,
\[
p_{ab}(x) = \frac{s(x) - s(a)}{s(b) - s(a)}
\]
and $D_s^+ p_{ab}(x)$ is constant.
As a result,
\begin{equation*}
\begin{split}
D^+_{s^h} p^h_{ab}(dx) &=  \frac{1} { h(b)}
[-p_{ab}(dx) D_s^+h(x)  -p_{ab}(x) D_s^+ h(dx)  +D_s^+ p_{ab}(dx)h(x)\\
&\quad  +D_s^+ p_{ab}(x)h(dx)]\\
&= \frac{1} { h(b)}
\left[-p_{ab}(dx) D_{s} h(x)  -p_{ab}(x) D_{s}h(dx)  +D_s^+p_{ab}(x)h(dx)\right]\\
&= -p_{ab}(x) D_s^+ h(dx)  \frac{1} { h(b)}  + \frac{1} { h(b)}
\left[-p_{ab}(dx) D_s^+ h(x)    +D_s^+ p_{ab}(x)h(dx)\right]\\
&= -p_{ab}(x) D_s^+ h(dx)  \frac{1} { h(b)}  \\
&\quad + \frac{1} { h(b)}
[-D_s^+ p_{ab}(x) D_s^+ h(x)s(dx)    +D_s^+ p_{ab}(x)D_s^+ h(x)s(dx)]\\
&= -p_{ab}(x) D_s^+ h(dx)  \frac{1} { h(b)}.
\end{split}
\end{equation*}
Thus,
\[
k^h(dx) = \frac{D^+_{s^h}p^h_{ab}(dx)}{p^h_{ab}(x)} = - \frac{h(x)}{h(a)h(b)}D_{s}^+ h(dx).
\]

The function $h$ restricted to the interval $(a,b)$ is excessive for the process $X$ killed when it exits $(a,b)$.
The $\alpha=0$ Green function for the latter process is the function $G_{ab}$ defined in \eqref{e:def_G_ab}, and
$h$ restricted to $(a,b)$ has a representation analogous to \eqref{e_aexc} of the form
%
\begin{equation}
\label{e_aexc_killed}
h(x) = h(a)\frac{s(b)-s(x)}{s(b)-s(a)} + h(b)\frac{s(x)-s(a)}{s(b)-s(a)}+ \int_a^b G_{ab}(x,y) G(x_0,y)^{-1} \, \nu(dy).
\end{equation}
Hence,
\[
\begin{split}
D_s^+ h(x)
& =  \frac{h(b) - h(a)}{s(b) - s(a)} \\
& \quad  -\int_{a < y \le x} r_0(x_0,y)^{-1}  \frac{s(y) - s(a)}{s(b) - s(a)} \nu(dy) \\
& \quad  +
\int_{x \le y < b} r_0(x_0,y)^{-1} \frac{s(b) - s(y)} {s(b) - s(a)} \nu(dy) \\
\end{split}
\]
and
\[
\begin{split}
D_s^+ h(dx)
& =
-r_0(x_0,x)^{-1}  \frac{s(x) - s(a)}{s(b) - s(a)} \nu(dx) \\
& \quad  -
r_0(x_0,x)^{-1} \frac{s(b) - s(x)}{s(b) - s(a)} \nu(dx) \\
& =
-r_0(x_0,x)^{-1}  \nu(dx). \\
\end{split}
\]

Thus,
\[
\begin{split}
k^h(dx)
& =
- \frac{D^+_{s^h}p^h_{ab}(dx)}{ p^h_{ab}(x)} \\
& =
-\frac{h(x) D_s^+h(dx) }{h(a) h(b)} \\
& =
\frac{1}{h(a) h(b)} h(x) r_0(x_0,x)^{-1} \,  \nu(dx) \\
\end{split}
\]
Note that this agrees with \eqref{e:h_killing} apart from the constant multiple $\frac{1}{h(a)h(b)}$.

Next, we turn to \eqref{e:speed_diffusion} to determine $m_{ab}^h$,
which we write as $m^h$.  Recall from \eqref{e_ab^h} that
$e^h_{ab}(x) = E_1(x) + E_2(x)$, $x \in (a,b)$, where
\[
E_1(x)
: =
\frac{1}{h(x)}\int_a^b G_{ab}(x,z) \int_a^b \frac{ G_{ab}(z,y)}{ r_0(x_0,y)} \, \nu(dy) \, m(dz)
\]
and
\begin{equation*}
\begin{split}
E_2(x)
:=
\frac{1}{h(x)}
&\Bigg[h(a) \int_a^b G_{a,b}(x,y) \frac{s(b)-s(y)}{s(b)-s(a)} \, m(dy)\\
&\quad +h(b) \int_a^b G_{a,b}(x,y) \frac{s(y)-s(a)}{s(b)-s(a)} \, m(dy) \Bigg].
\end{split}
\end{equation*}

We first need to compute
\[
D^+_{s^h}{e^h_{ab}}(x) := \lim_{\eta\downarrow x} \frac{e^h_{ab}(\eta)-e^h_{ab}(x)}{s^h(\eta)-s^h(x)}.
\]
If $a < x < y < b$, then
\[
D^+_{s^h}G_{a,b}(x,y) = \frac{h^2(x)}{h(a)h(b)}\frac{s(b)-s(y)}{s(b)-s(a)},
\]
while if $a < y < x < b$, then
\[
D^+_{s^h}G_{a,b}(x,y) = -\frac{h^2(x)}{h(a)h(b)}\frac{s(y)-s(a)}{s(b)-s(a)}.
\]
Thus,
\[
\begin{split}
D^+_{s^h}\left(\frac{G_{a,b}(x,y)}{h(x)}\right)
& =\frac{D^+_{s^h}G_{a,b}(x,y)}{h(x)} - \frac{G_{a,b}(x,y)D^+_{s^h}h(x)}{h^2(x)} \\
& = \frac{D^+_{s^h}G_{a,b}(x,y)}{h(x)} - \frac{G_{a,b}(x,y)D_s^+h(x)}{h(a)h(b)}. \\
\end{split}
\]

Now,
\begin{equation}
\begin{split}
D^+_{s^h}E_1(x)&=\int_a^b D^+_{s^h}\left(\frac{G_{ab}(x,z)}{h(x)}\right) \int_a^b \frac{ G_{ab}(z,y)}{ r_0(x_0,y)} \, \nu(dy) \, m(dz)\\
&=\int_a^b \biggl(\frac{D^+_{s^h}G_{a,b}(x,y)}{h(x)} \\
& \quad  - \frac{G_{a,b}(x,y)D_s^+h(x)}{h(a)h(b)}\biggr) \int_a^b \frac{ G_{ab}(z,y)}{ r_0(x_0,y)} \, \nu(dy) \, m(dz).\\
\end{split}
\end{equation}

Also,
\[
\begin{split}
D^+_{s^h}{E_2}(x) &= h(a) \int_a^b \biggl( \frac{D^+_{s^h}G_{a,b}(x,y)}{h(x)} \\
& \qquad - \frac{G_{a,b}(x,y)D_s^+h(x)}{h(a)h(b)}\biggr) \frac{s(b)-s(y)}{s(b)-s(a)} \, m(dy) \\
& \quad  +h(b) \int_a^b \biggl( \frac{D^+_{s^h}G_{a,b}(x,y)}{h(x)} \\
& \quad  \qquad - \frac{G_{a,b}(x,y)D_s^+h(x)}{h(a)h(b)}\biggr) \frac{s(y)-s(b)}{s(b)-s(a)} \, m(dy) \\
& = \int_a^x \biggl(\frac{-h(x)}{h(b)}\frac{s(y)-s(a)}{s(b)-s(a)} \\
& \qquad - \frac{ \frac{(s(y)-s(a))(s(b)-s(x))}{s(b)-s(a)}D_s^+h(x)}{h(b)} \biggr)    \frac{s(b)-s(y)}{s(b)-s(a)} \, m(dy) \\
& \quad  +  \int_x^b  \biggl( \frac{h(x)}{h(b)}\frac{s(b)-s(y)}{s(b)-s(a)} \\
& \quad  \qquad - \frac{ \frac{(s(x)-s(a))(s(b)-s(y))}{s(b)-s(a)}D_s^+h(x)}{h(b)}\biggr)           \frac{s(b)-s(y)}{s(b)-s(a)} \, m(dy) \\
& \quad  + \int_a^x\biggl(\frac{-h(x)}{h(a)}\frac{s(y)-s(a)}{s(b)-s(a)}  \\
& \quad  \qquad - \frac{ \frac{(s(y)-s(a))(s(b)-s(x))}{s(b)-s(a)}D_s^+h(x)}{h(a)} \biggr)  \frac{s(y)-s(a)}{s(b)-s(a)} \, m(dy) \\
& \quad  + \int_x^b \biggl( \frac{h(x)}{h(a)}\frac{s(b)-s(y)}{s(b)-s(a)} \\
& \quad  \qquad - \frac{ \frac{(s(x)-s(a))(s(b)-s(y))}{s(b)-s(a)}D_s^+h(x)}{h(a)}\biggr)   \frac{s(y)-s(a)}{s(b)-s(a)} \, m(dy).\\
\end{split}
\]

Next we need to identify the measure $D^+_{s^h}{e}^h(dx)$.  We have

\begin{align*}
D^+_{s^h}{E_2}(dx)
&= \int_a^x \Biggl(\frac{-D_s^+ h(x) s(dx)}{h(b)}\frac{s(y)-s(a)}{s(b)-s(a)} \\
& \qquad - \frac{ \frac{-s(dx)(s(y)-s(a))}{s(b)-s(a)}D_s^+h(x)}{h(b)}
- \frac{ \frac{(s(y)-s(a))(s(b)-s(x))}{s(b)-s(a)}D_s^+h(dx)}{h(b)}\Biggr) \\
& \qquad \times   \frac{s(b)-s(y)}{s(b)-s(a)}m(dy) \\
& \quad  +  \int_x^b  \Biggl( \frac{D_s^+h(x)s(dx)}{h(b)}\frac{s(b)-s(y)}{s(b)-s(a)}
- \frac{ \frac{(s(x)-s(a))(s(dx))}{s(b)-s(a)}D_s^+h(x)}{h(b)} \\
& \quad  \qquad - \frac{ \frac{(s(x)-s(a))(s(b)-s(y))}{s(b)-s(a)}D_s^+h(dx)}{h(b)}\Biggr)\\
& \quad  \qquad \times  \frac{s(b)-s(y)}{s(b)-s(a)} \, m(dy)\\
& \quad  + \int_a^x\Biggl(\frac{-D_s^+h(x)s(dx)}{h(a)}\frac{s(y)-s(a)}{s(b)-s(a)} \\
& \quad  \qquad - \frac{ \frac{-s(dx)(s(y)-s(a))}{s(b)-s(a)}D_s^+h(x)}{h(a)}
- \frac{ \frac{(s(y)-s(a))(s(b)-s(x))}{s(b)-s(a)}D_s^+h(dx)}{h(a)} \Biggr)  \\
& \quad  \qquad \times \frac{s(y)-s(a)}{s(b)-s(a)} \, m(dy) \\
& \quad  + \int_x^b \Biggl( \frac{D_s^+h(x)s(dx)}{h(a)}\frac{s(b)-s(y)}{s(b)-s(a)} \\
& \quad  \qquad - \frac{ \frac{(s(b)-s(y))(s(dx))}{s(b)-s(a)}D_s^+h(x)}{h(a)}
-\frac{ \frac{(s(x)-s(a))(s(b)-s(y))}{s(b)-s(a)}D_s^+h(dx)}{h(a)}\Biggr) \\
& \quad  \qquad \times  \frac{s(y)-s(a)}{s(b)-s(a)} \, m(dy)\\
& \quad  +\Biggl( \frac{-h(x)}{h(b)}\frac{s(x)-s(a)}{s(b)-s(a)}
- \frac{ \frac{(s(x)-s(a))(s(b)-s(x))}{s(b)-s(a)}D_s^+h(x)}{h(b)}\Biggr) \frac{s(b)-s(x)}{s(b)-s(a)} \, m(dx)\\
& \quad  -  \Biggl( \frac{h(x)}{h(b)}\frac{s(b)-s(x)}{s(b)-s(a)}
- \frac{ \frac{(s(x)-s(a))(s(b)-s(x))}{s(b)-s(a)}D_s^+h(x)}{h(b)}\Biggr)           \frac{s(b)-s(x)}{s(b)-s(a)} \, m(dx)\\
& \quad  + \Biggl(\frac{-h(x)}{h(a)}\frac{s(x)-s(a)}{s(b)-s(a)}
- \frac{ \frac{(s(x)-s(a))(s(b)-s(x))}{s(b)-s(a)}D_s^+h(x)}{h(a)} \Biggr)  \frac{s(x)-s(a)}{s(b)-s(a)} \, m(dx)\\
& \quad  -\Biggl( \frac{h(x)}{h(a)}\frac{s(b)-s(x)}{s(b)-s(a)}
- \frac{ \frac{(s(x)-s(a))(s(b)-s(x))}{s(b)-s(a)}D_s^+h(x)}{h(a)}\Biggr)   \frac{s(x)-s(a)}{s(b)-s(a)} \, m(dx).\\
\end{align*}

Doing the necessary cancellations results in

\begin{align*}
D^+_{s^h}{E_2}(dx)
&= \int_a^x \Bigg(- \frac{ \frac{(s(y)-s(a))(s(b)-s(x))}{s(b)-s(a)}D_s^+h(dx)}{h(b)}\Bigg) \frac{s(b)-s(y)}{s(b)-s(a)}m(dy) \\
& \quad  +  \int_x^b  \left( - \frac{ \frac{(s(x)-s(a))(s(b)-s(y))}{s(b)-s(a)}D_s^+h(dx)}{h(b)}\right)\frac{s(b)-s(y)}{s(b)-s(a)}m(dy)\\
& \quad  + \int_a^x\left(- \frac{ \frac{(s(y)-s(a))(s(b)-s(x))}{s(b)-s(a)}D_s^+h(dx)}{h(a)} \right) \frac{s(y)-s(a)}{s(b)-s(a)}m(dy) \\
& \quad  + \int_x^b \left(-\frac{ \frac{(s(x)-s(a))(s(b)-s(y))}{s(b)-s(a)}D_s^+h(dx)}{h(a)}\right)  \frac{s(y)-s(a)}{s(b)-s(a)}m(dy)\\
& \quad  +\left( \frac{-h(x)}{h(b)}\frac{s(x)-s(a)}{s(b)-s(a)} \right) \frac{s(b)-s(x)}{s(b)-s(a)}m(dx)\\
& \quad  -  \left( \frac{h(x)}{h(b)}\frac{s(b)-s(x)}{s(b)-s(a)}\right)           \frac{s(b)-s(x)}{s(b)-s(a)}m(dx)\\
& \quad  + \left(-\frac{h(x)}{h(a)}\frac{s(x)-s(a)}{s(b)-s(a)} \right)  \frac{s(x)-s(a)}{s(b)-s(a)}m(dx)\\
& \quad  -\left( \frac{h(x)}{h(a)}\frac{s(b)-s(x)}{s(b)-s(a)} \right)   \frac{s(x)-s(a)}{s(b)-s(a)}m(dx)\\
& = -D_s^+ h(dx)h(x) E_2(x)\frac{1}{h(a)h(b)}-\biggl(\frac{h(x)}{h(b)}\frac{s(b)-s(x)}{s(b)-s(a)} \\
& \qquad + \frac{h(x)}{h(a)}\frac{s(x)-s(a)}{s(b)-s(a)}\biggr) \, m(dx)\\
&= -D_s^+ h(dx)h(x) E_2(x)\frac{1}{h(a)h(b)} \\
& \quad  - \left(\frac{h(x)}{h(b)}\frac{s(b)-s(x)}{s(b)-s(a)} + \frac{h(x)}{h(a)}\frac{s(x)-s(a)}{s(b)-s(a)}\right) \, m(dx).\\
\end{align*}

Similar computations for $E_1$ give
\[
\begin{split}
D^+_{s^h}{E_1}(dx)
&= \left(-\frac{h(x)}{h(a)h(b)}\frac{s(x)-s(a)}{s(b)-s(a)}\right) \int_a^b\frac{ G_{ab}(x,y)}{ r_0(x_0,y)} \, \nu(dy)\,m(dx) \\
& \quad  -\left(\frac{h(x)}{h(a)h(b)}\frac{s(b)-s(x)}{s(b)-s(a)}\right) \int_a^b\frac{ G_{ab}(x,y)}{ r_0(x_0,y)} \, \nu(dy)\,m(dx) \\
& \quad  -\frac{1}{h(a)h(b)}h(x) D_s^+h(dx) E_1(x)\\
&= -\frac{h(x)}{h(a)h(b)}\int_a^b\frac{ G_{ab}(x,y)}{ r_0(x_0,y)} \, \nu(dy)\,m(dx) -\frac{1}{h(a)h(b)}h(x) D_s^+h(dx) I_1(x).\\
\end{split}
\]

Thus,
\[
\begin{split}
D^+_{s^h}e^h(dx)
&= -D_s^+ h(dx)h(x) E_2(x)\frac{1}{h(a)h(b)} \\
& \quad  - \left(\frac{h(x)}{h(b)}\frac{s(b)-s(x)}{s(b)-s(a)} + \frac{h(x)}{h(a)}\frac{s(x)-s(a)}{s(b)-s(a)}\right) \, m(dx)\\
& \quad  -\frac{h(x)}{h(a)h(b)}\int_a^b\frac{ G_{ab}(x,y)}{ r_0(x_0,y)} \, \nu(dy)\,m(dx) -\frac{1}{h(a)h(b)}h(x) D_s^+h(dx) E_1(x)\\
&=  -D_s^+ h(dx)h(x) e^h(x) \frac{1}{h(a)h(b)} \\
& \quad  -\frac{h(x)}{h(a)h(b)}\biggl(h(a)\frac{s(b)-s(x)}{s(b)-s(a)} + h(b)\frac{s(x)-s(a)}{s(b)-s(a)} \\
& \qquad +\int_a^b\frac{ G_{ab}(x,y)}{ r_0(x_0,y)} \, \nu(dy)\biggr) \, m(dx)\\
&= e^h(x)k^h(dx) -\frac{h(x)}{h(a)h(b)}\biggl(h(a)\frac{s(b)-s(x)}{s(b)-s(a)} + h(b)\frac{s(x)-s(a)}{s(b)-s(a)} \\
& \quad  +\int_a^b\frac{ G_{ab}(x,y)}{ r_0(x_0,y)} \, \nu(dy)\biggr) \,m(dx).\\
\end{split}
\]

Substituting the above computations into \eqref{e:speed_diffusion} produces
\begin{equation}\label{e_m^h}
\begin{split}
m^h(dx) &=- [D^+_{s^h}e^h(dx) - e^h(x)k^h_{a,b}(dx)]\\
&=\frac{h(x)}{h(a)h(b)} \biggl(\int_a^b\frac{ G_{ab}(x,y)}{ r_0(x_0,y)} \, \nu(dy) \\
& \quad  +h(a)\frac{s(b)-s(x)}{s(b)-s(a)} + h(b)\frac{s(x)-s(a)}{s(b)-s(a)}\biggr) \, m(dx).\\
\end{split}
\end{equation}
Combining, \eqref{e_m^h} and \eqref {e_aexc_killed} gives
\[
m^h(dx) = \frac{1}{h(a) h(b)} h^2(x) \, m(dx).
\]
Note that this agrees with \eqref{e:h_speed} apart from the constant multiple $\frac{1}{h(a)h(b)}$.

Lastly, note that for a nonnegative function $f:I \to \bR$, we have
\[
\begin{split}
\int_I r_0^h(x,y) f(y) \, m^h(dy)
& =
\int_0^\infty \int_I f(y) P_t^h(x,dy)  \, dt \\
& =
\int_0^\infty \int_I f(y) \frac{1}{h(x)} h(y) P_t(x,dy)  \, dt \\
& =
\int_I f(y) \frac{1}{h(x)} h(y) r_0(x,y) \, m(dy) \\
& =
\int_I f(y) \frac{1}{h(x)} h(y) r_0(x,y) h(y)^{-2} \, m^h(dy) \\
& =
\int_I f(y) \frac{1}{h(x) h(y)} r_0(x,y) \, m(dy), \\
\end{split}
\]
and so
\[
r_0^h(x,y) = \frac{1}{h(x) h(y)} r_0(x,y),
\]
as required.

This completes the proof of Theorem~\ref{t:h_diffusion_chars}.
\end{proof}

\begin{remark}
The characteristics $s^h, k^h, m^h$ of the $h$-transformed process seem to be known in some degree of generality in the folklore. 
We presented a proof because we were not able to find a sufficiently general result in the literature. 
We assumed that the original, unconditioned process $X$ does not have killing, $k\equiv 0$, because this is the case
that is of interest to us and including killing would complicate the computations. 
See \cite{LS90} for results along the lines of ours under certain assumptions.
\end{remark}

\section{Doob $h$-transforms for one-dimensional diffusions: generators}\label{s:doob_generator}

The diffusion $X$ determines and in turn is determined by its infinitesimal generator.
The infinitesimal generator is specified by the scale, speed and killing measures
and by {\em boundary conditions} on functions in the domain.  For the sake of completeness, following \cite{Bor02},
we now sketch this correspondence.

Fix $z\in I$.
The left-hand point $\ell$ is called \textit{exit} for $X$ if
\[
\int_{(\ell,z)}[m((x,z))+k((x,z))] \, s(dx)<\infty
\]
and \textit{entrance} if
\[
\int_{(\ell,z)}(s(z)-s(x)) \, (m(dx)+k(dx))<\infty
\]
with similar definitions for the right-hand point $r$.
A boundary point that is both entrance and exit is called \textit{non-singular} or \textit{regular}.

A boundary point which is neither entrance nor exit is called a \textit{natural} boundary.
If $\ell$ is natural, then it is said to be \textit{attractive} if $\lim_{x\downarrow\ell} s(x)>-\infty$.
In this case we have $\lim_{t\rightarrow \infty}X_t=\ell$ with positive probability.

\begin{definition}\label{d:generator}
The \textit{(weak) infinitesimal generator} of $X$ is the operator $\cG^{\bullet}$ defined by
\[
\cG^{\bullet}f:=\lim_{t\downarrow 0}\frac{P_t f-f}{t}
\]
applied to $f\in \cC_b(I)$ for which the limit exists pointwise, is in $\cC_b(I)$, and
\[
\sup_{t>0}\left\|\frac{P_t f-f}{t}\right\|<\infty.
\]
Denote by $\cD(\cG^{\bullet})$ the set of such functions. Define a set of functions $\cD(\cG)$ by saying that $f\in \cC_b(I)$
belongs to $\cD(\cG)$ if $D_{s}^-f$ and $D_{s}^+f$ exist
and there exists a function $g\in \cC_b(I)$ such that for all $\ell<a<b<r$,
\begin{itemize}
  \item [(a)]
  \[
  \int_{[a,b)}g(x) \, m(dx) = D_{s}^- f(b) - D_{s}^- f(a) - \int_{[a,b)}f(x) \, k(dx).
  \]
   \item [(b)]
  \[
  \int_{(a,b]}g(x) \, m(dx) = D_{s}^+ f(b) - D_{s}^+ f(a) - \int_{(a,b]}f(x) \, k(dx).
  \]
  \item [(c)] If $\ell$ is regular and $m(\{\ell\}), k(\{\ell\})<\infty$
  \[
  g(\ell)m(\{\ell\}) = D_{s}^+f(\ell+) - f(\ell)k(\{\ell\}).
  \]
  \item [(d)] If $r$ is regular and $m(\{r\}), k(\{r\})<\infty$
  \[
  g(r)m^h(\{r\}) = -D_{s}^-f(r-) - f(r)k(\{r\}).
  \]
  \item [(e)] If $\ell$ is entrance-not-exit
  \[
  D_{s}^+f(\ell+)=0.
  \]
  \item [(f)] If $\ell$ is exit-not-entrance
  \[
  f(\ell+)=0.
  \]
  \item [(g)] If $r$ is entrance-not-exit
  \[
  D_{s}^-f(r-)=0.
  \]
  \item [(h)] If $r$ is exit-not-entrance
  \[
  f(r-)=0.
  \]
  \item [(i)] If $m(\{\ell\})=\infty$ or $k(\{\ell\})=\infty$, then
  \[
  g(\ell)=-\gamma(\ell)f(\ell), \;\gamma(\ell)>0.
  \]
\end{itemize}
Define
\[
\cG f:= g
\]
for $f\in \cD(\cG)$. Note that (a) and (b) imply that when $f\in \cD(\cG)$ then
$D_{s}^-f$ is left continuous and $D_{s}^+f$ right continuous.
From \cite{IM96}
one has
\[
\cG=\cG^{\bullet},\quad \cD(\cG) = \cD(\cG^{\bullet}).
\]
\end{definition}

Let $\cG$ be the generator of a diffusion $X$ on $I:=(\ell,r)$ where $\ell$ and $r$ are inaccessible.
Suppose $u$ is a continuous solutions to the ODE
\begin{equation}\label{e:G_psi}
\cG u =\alpha u
\end{equation}
that is,
\be\label{e:psi_phi}
\alpha\int_{[a,b)}u(x) \, m(dx)=D_s^-u(b)-D_s^-u(a)-\int_{[a,b)}u(x) \,k(dx)
\ee
for all $(a,b)\subset I$. For $\alpha> 0$ the functions $\psi_\alpha$ and $\phi_\alpha$ from \eqref{e:psi_+} and \eqref{e:psi_-} can be characterized as the unique (up to a multiplicative constant) solutions of \eqref{e:psi_phi} by firstly demanding that $\psi_\alpha$ is increasing and $\phi_\alpha$ decreasing, and then imposing the boundary conditions
\begin{eqnarray*}
\psi_\alpha(\ell +) = \phi_\alpha(r-)=0,
\end{eqnarray*}
and
\[
\psi_\alpha (r-) =\phi_\alpha(\ell+)=+\infty.
\]

\begin{remark}
Consider the special case where the diffusion $X$ has null killing measure and scale and speed measures that
are absolutely continuous with respect to Lebesgue measure
\begin{itemize}
  \item  $m(dx)=m'(x) \, dx$.
  \item  $s(dx)=s'(x) \, dx$.
  \item  $k\equiv 0$.
\end{itemize}
If $s'\in C^1(I)$ then the infinitesimal generator $\cG:\cD(\cG)\to\cC_b(I)$ of $X$ is a second order differential operator
\[
\cG f(x) = \frac{1}{2}\sigma^2(x)\partial_{xx}f(x) + b(x)\partial_x f(x)
\]
where
\be\label{e:scale_speed_C^1}
m'(x) = 2\sigma^{-2}(x)e^{B(x)}, s'(x)=e^{-B(x)}
\ee
with $B(x):=\int^x2\sigma^{-2}(y)b(y)\,dy$. The domain $\cD(\cG)$ consists of all functions in $\cC_b(I)$ such that $\cG f\in\cC_b(I)$ together with the appropriate boundary conditions.
\end{remark}

\begin{remark}
If $m$ is absolutely continuous with respect to Lebesgue measure, $m(dx)=m'(x)dx$ then
\begin{equation}\label{e:density_speed}
p(t;x,y) = q(t;x,y)/m'(y)
\end{equation}
where $q(t;x,y)$ is the transition density with respect to Lebesgue measure.
\end{remark}


We follow \cite{Bor02} and \cite{IM96} in order to characterize the generator of the $h$-transformed diffusion.

We showed that if we have a transient diffusion $X$ on $I=(\ell,r)$ with natural boundary points $\ell$ and $r$, that is characterized by a scale measure $s(dx)$ and a speed measure $m(dx)$ and no killing, then, if $h$ is excessive with representation
\be\label{e:h_representation}
h(x)= \int_{(\ell,r)} \frac{r_0(x,y)}{r_0(x_0,y)} \nu(dy) + \frac{\phi_0(x)}{\phi_0(x_0)} \nu(\{\ell\}) +\frac{\psi_0(x)}{\psi_0(x_0)} \nu(\{r\})
\ee
the $h$-transform $X^h$ is a diffusion on $I$ that is characterized by
\begin{itemize}
\item Speed measure
 \begin{equation*}
 m^h(dy) = h^2(y) m(dy).
 \end{equation*}
 \item Scale function
  \begin{equation*}
 s^h(dy) = h^{-2}(y) s(dy).
 \end{equation*}
 \item Killing measure
  \begin{equation*}
 k^h(dy)= (G^h(x_0,y))^{-1}\nu(dy), y\in I, G^h:=\frac{r_0(x,y)}{h(x)h(y)}.
 \end{equation*}
 \end{itemize}
We can now write down Definition \ref{d:generator} for the process $X^h$.
The (weak) infinitesimal generator of $X^h$ is the operator $\cG^{h,\bullet}$ defined by
\[
\cG^{h,\bullet}:=\lim_{t\downarrow 0}\frac{P^h_t f-f}{t}
\]
applied to $f\in \cC_b(I)$ for which the limit exists pointwise, is in $\cC_b(I)$, and
\[
\sup_{t>0}\left\|\frac{P^h_t f-f}{t}\right\|<\infty.
\]
Denote by $\cD(\cG^{h,\bullet})$ the set of such functions. Define a set of functions $\cD(\cG^h)$ by saying that $f\in \cC_b(I)$ belongs to $\cD(\cG^h)$ if $D_{s^h}^-f$ and $D_{s^h}^+f$ exist and there exists a function $g\in \cC_b(I)$ such that for all $\ell<a<b<r$,
\begin{itemize}
  \item [(a)]
  \[
  \int_{[a,b)}g(x)m^h(dx) = D_{s^h}^- f(b) - D_{s^h}^- f(a) - \int_{[a,b)}f(x)k^h(dx).
  \]

  \item [(b)]
  \[
  \int_{(a,b]}g(x)m^h(dx) = D_{s^h}^+ f(b) - D_{s^h}^+ f(a) - \int_{(a,b]}f(x)k^h(dx).
  \]
  \item [(c)] If $\ell$ is regular and $m^h(\{\ell\}), k^h(\{\ell\})<\infty$
  \[
  g(\ell)m^h(\{\ell\}) = D_{s^h}^+f(\ell+) - f(\ell)k^h(\{\ell\}).
  \]
  \item [(d)] If $r$ is regular and $m^h(\{r\}), k^h(\{r\}<\infty$
  \[
  g(r)m^h(\{r\}) = -D_{s^h}^-f(r-) - f(r)k^h(\{r\}).
  \]
  \item [(e)] If $\ell$ is entrance-not-exit
  \[
  D_{s^h}^+f(\ell+)=0.
  \]
  \item [(f)] If $\ell$ is exit-not-entrance
  \[
  f(\ell+)=0.
  \]
  \item [(g)] If $r$ is entrance-not-exit
  \[
  D_{s^h}^-(r-)=0.
  \]
  \item [(h)] If $r$ is exit-not-entrance
  \[
  f(r-)=0.
  \]
  \item [(i)] If $m^h(\{\ell\})=\infty$ and/or $k^h(\{\ell\})=\infty$ then
  \[
  g(\ell)=-\gamma(l)f(l), ~\gamma(\ell)>0.
  \]
\end{itemize}
Define
\[
\cG^h f:= g
\]
for $f\in \cD(\cG^h)$. Note that (a) and (b) imply that when $f\in \cD(\cG^h)$
then $D_{s^h}^-f$ is left continuous and $D_{s^h}^+f$ right continuous.
By \cite{IM96} pages 100, 117 and 135 one has
\[
\cG^h=\cG^{h,\bullet},\quad \cD(\cG^h) = \cD(\cG^{h,\bullet}).
\]

Consider the special case when the diffusion $X$ has
\begin{itemize}
  \item Speed measure $m(dx)=m'(x)dx$.
  \item Scale function $s(x)=\int^x s'(y),dy$.
  \item No killing $k\equiv 0.$
\end{itemize}
Then $m^h, s^h, k^h$ are given by \eqref{e:h_speed}, \eqref{e:h_scale} and \eqref{e:h_killing}. Equations (a) and (b) above become
\be\label{e:integral_f^-}
\int_{[a,b)}g(x)h^2(x) m'(x)\,dx = \frac{h^2(b)}{s'(b)}f^-(b) - \frac{h^2(a)}{s'(a)} f^-(a) - \int_{[a,b)}f(x)k^h(dx).
\ee
\be\label{e:integral_f^+}
\int_{(a,b]}g(x)h^2(x)m'(x)\,dx =\frac{h^2(b)}{s'(b)}f^+(b) - \frac{h^2(a)}{s'(a)} f^+(a) - \int_{(a,b]}f(x)k^h(dx).
\ee
In order to find the representation of $h(x)$ from \eqref{e:h_representation} note that
\[
r_0(x,y) = \left\{
             \begin{array}{ll}
               c_0\psi_0(x)\phi_0(y), & \hbox{$x\leq y$;} \\
               c_0\psi_0(y)\phi_0(x), & \hbox{$x\geq y$.}
             \end{array}
           \right.
\]
where $\psi_0,\phi_0$ are the functions from \eqref{e:psi_0_+}, \eqref{e:psi_0_-} and
\[
c_0^{-1}
=
\phi_0(x) D_s^+ \psi_0(x) - \psi_0(x) D_s^+ \phi_0(x).
\]

Suppose that the original process wanders off to $\ell$.  Then
\[
\Pr^x\{T_a < T_b\} = (s(b) - s(x)) / (s(b) - s(a)).
\]

Note that
\[
\Pr^x\{T_z < \infty\}
= \left\{
    \begin{array}{ll}
      \psi_0(x) / \psi_0(z), & \hbox{$x \le z$,} \\
      \phi_0(x) / \phi_0(z), & \hbox{$x \ge z$.}
    \end{array}
  \right.
\]
Now for $x \le z$,
\[
\Pr^x\{T_z < \infty\}=\lim_{a \downarrow \ell} \Pr^x\{T_z < T_a\}=\lim_{a \downarrow \ell} (s(x) - s(a))/(s(z) - s(a)),
\]
while for $x \ge z$
\[
\Pr^x\{T_z < \infty\} = 1.
\]
This shows that we should take
\[
\psi_0(x) = \lim_{a \to \ell} (s(x) - s(a))
\]
and
\[
\phi_0(x) = 1.
\]
We can assume that
\[\lim_{a \to \ell} s(a) = 0.
\]
With that assumption,
\[
r_0(x,y)
= \left\{
    \begin{array}{ll}
      c_0 s(x), & \hbox{$x \le y$,} \\
      c_0 s(y), & \hbox{$x \ge y$.}
    \end{array}
  \right.
\]
We have
\[
c_0^{-1}
=
\phi_0(x) D_s^+ \psi_0(x) - \psi_0(x) D_s^+ \phi_0(x)
=
1 \times 1 - s(x) \times 0
\]
and so $c_0 = 1$.  Therefore,
\be\label{e:r_0_ell}
r_0(x,y)
= \left\{
    \begin{array}{ll}
      s(x), & \hbox{$x \le y$,} \\
       s(y), & \hbox{$x \ge y$.}
    \end{array}
  \right.
\ee

\section{Bang-bang process of a one-dimensional diffusion}\label{s:prop}

Assume that $X$ is a one-dimensional diffusion with state space $I$.
Using the formula for the resolvent of $X^b$, namely equation \eqref{e:resolve_second_construction2}, we get that with respect to the speed measure $m$ of $X$ the resolvent of $X^b$ has densities
\begin{equation}\label{e:r_lambda_1}
\begin{split}
r_\lambda^b(x,y)&= R_\lambda^b \delta_y(x)\\
&= {_a}R_\lambda^b \delta_y(x) + \psi_\lambda^b(x)R_\lambda^b\delta_y(a)\\
&= R_\lambda^h\delta_y(x) - \frac{R_\lambda^h\delta_a(x)}{R_\lambda^h\delta_a(a)}R_\lambda^h\delta_y(a) + h(x)^{-1}\frac{r_\lambda(x,a)}{r_\lambda(a,a)}\frac{R_\lambda^h\delta_y(a)}{\lambda R_\lambda^h(1)(a)} \\
&= r_\lambda^h(x,y) - \frac{r_\lambda^h(x,a)}{r_\lambda^h(a,a)}r_\lambda^h(a,y) + \frac{r_0(a,a)}{r_0(x,a)}\frac{r_\lambda(x,a)}{r_\lambda(a,a)} \frac{r_\lambda^h(a,y) r_0(a,a)}{r_0(a,a)-r_\lambda(a,a)}\\
\end{split}
\end{equation}
Note that with respect to the measure $m$ the $h$-transform looks like
\begin{equation}\label{e:r^h}
r_\lambda^h(x,y) = \frac{r_\lambda(x,y)}{h(x)h(y)} h^2(y) = r_\lambda(x,y) \frac{h(y)}{h(x)} = r_\lambda(x,y) \frac{r_0(y,a)}{r_0(x,a)}.
\end{equation}
As a result of \eqref{e:r_lambda_1} and \eqref{e:r^h}
\begin{eqnarray*}
r_\lambda^b(x,y) &=& r_\lambda(x,y)\frac{r_0(y,a)}{r_0(x,a)} - \frac{r_\lambda(x,a)\frac{r_0(a,a)}{r_0(x,a)}}{r_\lambda(a,a)}r_\lambda(a,y)\frac{r_0(y,a)}{r_0(a,a)} \\
&~&~+ ~\frac{r_0(a,a)}{r_0(x,a)}\frac{r_\lambda(x,a)}{r_\lambda(a,a)} \frac{r_\lambda(a,y)\frac{r_0(y,a)}{r_0(a,a)} r_0(a,a)}{r_0(a,a)-r_\lambda(a,a)}\nonumber\\
&=&r_\lambda(x,y)\frac{r_0(y,a)}{r_0(x,a)} - \frac{r_\lambda(x,a)}{r_\lambda(a,a)}r_\lambda(a,y)\frac{r_0(y,a)}{r_0(x,a)} \\
&~&\quad + \frac{r_0(a,a)}{r_0(x,a)}\frac{r_\lambda(x,a)}{r_\lambda(a,a)} \frac{r_\lambda(a,y) r_0(y,a) }{r_0(a,a)-r_\lambda(a,a)}\nonumber\\
&=& r_\lambda(x,y)\frac{r_0(y,a)}{r_0(x,a)} + \frac{r_\lambda(x,a)}{r_\lambda(a,a)}r_\lambda(a,y)\frac{r_0(y,a)}{r_0(x,a)}\left(\frac{r_\lambda(a,a)}{r_0(a,a)-r_\lambda(a,a)}\right)\nonumber
\end{eqnarray*}

and therefore
\begin{equation}\label{e:r^b}
r_\lambda^b(x,y)=\frac{r_0(y,a)}{r_0(x,a)} \left[r_\lambda(x,y)+ \frac{r_\lambda(x,a) r_\lambda(a,y)}{r_0(a,a)-r_\lambda(a,a)}\right]
\end{equation}

\begin{remark}
\label{r:bangbang_symmetric}
The resolvent of $X^b$ has symmetric densities
\[
\frac{r_0^2(a,a)}{r_0(x,a) r_0(y,a)} \left[r_\lambda(x,y)+ \frac{r_\lambda(x,a) r_\lambda(a,y)}{r_0(a,a)-r_\lambda(a,a)}\right]
\]
with respect to the measure 
$\left(\frac{r_0(y,a)}{r_0(a,a)}\right)^2 \, m(dy) = h^2(y) \, m(dy) = m^h(dy)$.
It follows that $m^h$ is a multiple of the speed measure of $X^b$.
\end{remark}

\begin{example}
Suppose that $X$ is Brownian motion with drift $-\mu$, $\mu > 0$, and $a=0$.  
For a suitable normalization of the scale measure, the speed measure of $X$ is 
$2 \exp(2 \mu x) \, dx$ and the corresponding resolvent densities are
$r_0(x,y) = 2 \mu \exp(-2 \mu(x \vee y))$ (see, for example, Appendix 1.14 in \cite{Bor02}).
We can use Remark~\ref{r:bangbang_symmetric} in a simple but somewhat tedious calculation
to compute the resolvent densities of $X^b$ against
the measure $m^h(dy) = 2 \exp(-2 \mu |y|)$ and see that they agree with the resolvent densities
of bang-bang Brownian motion given in Appendix 1.15 of \cite{Bor02}, so that
$X^b$ is indeed bang-bang Brownian motion.
\end{example}


\begin{example}
Let $X$ be the Ornstein-Uhlenbeck process
\begin{equation}\label{e:OU}
dX_t = -\gamma X_t \,dt + dW_t.
\end{equation}
The speed measure of this process is
\begin{equation}\label{e:speed}
m_\gamma(dx) = 2 \exp(-\gamma x^2)\,dx.
\end{equation}

When $\gamma>0$ the process is positive recurrent while when $\gamma<0$ the process is transient.
Suppose from now on that $\gamma<0$ so that we are in the transient case. We want to see what the process $X^b$ is in this setting. 

From  \cite{Bor02} Appendix 1.24 we have that the resolvent density of $X$ with respect to $m_\gamma$ is
\begin{equation}\label{e:r_lambda_OU}
\begin{split}
r_\lambda(x,y)\\
&:= \frac{\Gamma(\lambda/|\gamma|+1)}{2\sqrt{|\gamma|\pi}} \exp\left(-\frac{|\gamma|x^2}{2}\right) D_{-\lambda/|\gamma|-1}(-x\sqrt{2|\gamma|})\\
&\quad \times \exp\left(-\frac{|\gamma|y^2}{2}\right)D_{-\lambda/|\gamma|-1}(y\sqrt{2|\gamma|}),\quad x\geq y\\
\end{split}
\end{equation}
where $\Gamma(x)$ is the Gamma function and
\begin{eqnarray*}
D_{-\nu}(x) &:=&  e^{-x^2/4}2^{-\nu/2}\sqrt{\pi}\Bigg\{\frac{1}{\Gamma((\nu+1)/2)}\left(1+\sum_{k=1}^\infty\frac{\nu(\nu+2)\cdots(\nu+2k-2)}{(2k)!}x^{2k}\right)\\
&~&~-~\frac{x\sqrt{2}}{\Gamma(\nu/2)}\left(1+\sum_{k=1}^\infty\frac{(\nu+1)(\nu+3)\cdots(\nu+2k-1)}{(2k+1)!}x^{2k}\right)\Bigg\}
\end{eqnarray*}
is the parabolic cylinder function.

A natural conjecture would be that $X^b$ is a recurrent OU process. 
We show that this is not the case. Set $a=0$ and $y=0$. Then, for $x\geq 0$, equation \eqref{e:r_lambda_OU} becomes
\begin{equation}\label{e:r_0OU}
\begin{split}
r_0(x,0) &= \frac{\Gamma(1)}{2\sqrt{|\gamma|\pi}} \exp\left(-\frac{|\gamma|x^2}{2}\right) D_{-1}(|-x\sqrt{2|\gamma|}|)D_{-1}(0)\\
&= \frac{\Gamma(1)}{2\sqrt{|\gamma|\pi}} \exp\left(-\frac{|\gamma|x^2}{2}\right)  e^{(-x\sqrt{2|\gamma|})^2/4}\sqrt{\frac{\pi}{2}} \text{erfc}\left(\frac{|x|\sqrt{2|\gamma|}}{\sqrt 2}\right) \sqrt{\frac{\pi}{2}} \text{erfc}\left(0\right)\\
&= \frac{1}{4}\sqrt{\frac{\pi}{|\gamma|}} \text{erfc}(|x|\sqrt{|\gamma|}).\\
\end{split}
\end{equation}
where we used the identity
\begin{equation*}
D_{-1}(x) = e^{x^2/4}\sqrt{\frac{\pi}{2}} \text{erfc}\left(\frac{|x|}{\sqrt 2}\right).
\end{equation*}
and the \textit{error function} $\text{erf}$ and the \textit{complementary error function} $\text{erfc}$ are defined via
\begin{equation*}
\text{erf}(x) = \frac{2}{\sqrt \pi} \int_0^x e^{-t^2}\,dt
\end{equation*}
and
\begin{equation*}
\text{erfc}(x) = 1 - \text{erf}(x).
\end{equation*}
As a result of \eqref{e:r_0OU}
\begin{equation}\label{e:h_OU}
h(x) = \frac{r_0(x,0)}{r_0(0,0)} = \text{erfc}(|x|\sqrt{|\gamma|}).
\end{equation}

From Remark~\ref{r:bangbang_symmetric}, the speed measure of $X^b$ is a multiple of
\begin{eqnarray*}
m_\gamma^h (dx) &=& h^2(x) m_\gamma(dx)\\
&=& \left(\text{erfc}(|x|\sqrt{|\gamma|})\right)^2 2 \exp(-\gamma x^2)\,dx.
\end{eqnarray*}
Such a measure does not look like $m_{\gamma^*}$ from \eqref{e:speed} for any $\gamma^*$ and hence $X^b$ is not an OU process.
\end{example}

\section{Generator of the conditioned diffusion}\label{s:fixedpoint}

\begin{theorem}\label{t:cond_point}
Let $X$ be a one-dimensional transient diffusion on $I=(\ell,r)$ with $\ell,r$ inaccessible boundary points and such that
\[
\lim_{t\rightarrow \infty} X_t = \ell
\]
$\Pr^x$ almost surely for all $x\in (\ell,r)$. Assume that
\begin{itemize}
\item $X$ has an absolutely continuous speed measure $m(dx)=m'(x)\,dx$ and scale function $s(dx)=s'(x)\,dx$.
\item The densities $s'$ and $m'$ are strictly positive on $(\ell,r)$.
\item The densities are smooth enough, namely  $s'\in C^1((\ell,r))$ and $m'\in C((\ell,r))$.
\end{itemize}
Set 
\[
h(x)= \Pr^x\{T_a<\infty\}, \quad  x \in I.
\]

The generator $\cG^h$ of $X^h$ is given by
\[
\cG^h f(x) = \left\{
               \begin{array}{ll}
                 \frac{1}{h^2(y)m'(y)}\left(\frac{h^2(y)}{s'(y)}f'(y)\right)', & \hbox{$y\neq a$;} \\
                 \frac{-s''(a)}{m'(a)(s'(a))^2}f^+(a) + \frac{1}{m'(a)s'(a)}(f^+)^+(a), & \hbox{$y= a$;} \\
               \end{array}
             \right.
\]
and the domain of the generator is
\begin{equation*}
\begin{split}
&\cD(\cG^h)\\
&=\Bigg\{f\in C^2(\ell,a)\cap C^2(a,r): f^+(a)-f^+(a-)=f^-(a+)- f^-(a)=\frac{s'(a)}{s(a)}f(a),\\
& \frac{2s(a)(s'(a))^2-s^2(a)s''(a)}{(s'(a))^2}f^-(a) + \frac{s^2(a)}{s'(a)}(f^-)^-(a)=\frac{-s^2(a)s''(a)}{(s'(a))^2}f^+(a) \\
&\quad + \frac{s^2(a)}{s'(a)}(f^+)^+(a) \Bigg\}.
\end{split}
\end{equation*}
\end{theorem}

\begin{proof}
Take
\[
h(x) = r_0(x,a) = \Pr^x\{T_a<\infty\} r_0(a,a).
\]
So, by \eqref{e:r_0_ell}
\be\label{e:h_point}
h(x)= \left\{
    \begin{array}{ll}
      s(x), & \hbox{$x \le a$,} \\
       s(a), & \hbox{$x \ge a$.}
    \end{array}
  \right.
\ee
Thus,
\be\label{e:h'_point}
h'(x)= \left\{
    \begin{array}{ll}
      s'(x), & \hbox{$x < a$,} \\
      0, & \hbox{$x > a$.}
    \end{array}
  \right.
\ee
At $x=a$ one has
\[
h^-(a)=s'(a)
\]
together with
\[
h^+(a)=0.
\]
It is clear from \eqref{e:h_representation} and the definition of $h$ that
\[
k^h(dx) = r_0(a,a) \delta_a(x) = s(a)\delta_a(dx)
\]
For $\ell<u<v<r$ and $a\notin [u,v]$ equations \eqref{e:integral_f^-} and  \eqref{e:integral_f^+} become
\[
\int_{[u,v)}g(x)h^2(x) m'(x)\,dx = \frac{h^2(u)}{s'(u)}f^-(u) - \frac{h^2(v)}{s'(v)} f^-(v)
\]
and
\[
\int_{(u,v]}g(x)h^2(x)m'(x)\,dx =\frac{h^2(u)}{s'(u)}f^+(u) - \frac{h^2(v)}{s'(v)} f^+(v)
\]
which imply by arguments similar to the above that $f\in C^2(u,v)$ and
\[
g(x)h^2(x)m'(x) = \left(\frac{h^2(x)}{s'(x)}f'(x)\right)'
\]
for all $x\in(u,v)$.

Now use \eqref{e:integral_f^-} for the interval $[a,a+\epsilon)$ to get
\be
\begin{split}
\int_{[a,a+\epsilon)}g(x)h^2(x) m'(x)\,dx &= \frac{h^2(a+\epsilon)}{s'(a+\epsilon)}f^-(a+\epsilon) - \frac{h^2(a)}{s'(a)} f^-(a) \\
&\quad - \int_{[a,a+\epsilon)}f(x)s(a)\delta_a(dx)\\
&= \frac{h^2(a+\epsilon)}{s'(a+\epsilon)}f^-(a+\epsilon) - \frac{h^2(a)}{s'(a)} f^-(a) - f(a)s(a)
\end{split}
\ee
which implies
\be\label{e:f^-(a+)}
\left(\lim_{\epsilon\downarrow 0}f^-(a+\epsilon)-f^-(a)\right) = \frac{s'(a)}{h^2(a)}f(a)s(a) = \frac{s'(a)}{s(a)} f(a)
\ee
Similarly if we use \eqref{e:integral_f^-} for the interval $[a-\epsilon,a)$
\be
\begin{split}
\int_{[a-\epsilon,a)}g(x)h^2(x) m'(x)\,dx &= \frac{h^2(a)}{s'(a)}f^-(a) - \frac{h^2(a-\epsilon)}{s'(a-\epsilon)} f^-(a-\epsilon) \\
&\quad - \int_{[a-\epsilon,a)}f(x)s(a)\delta_a(dx)\\
&= \frac{h^2(a)}{s'(a)}f^-(a) - \frac{h^2(a-\epsilon)}{s'(a-\epsilon)} f^-(a-\epsilon)
\end{split}
\ee
which forces
\be
\begin{split}
g(a)h^2(a) m'(a)&= \lim_{\epsilon\downarrow 0}\frac{\frac{h^2(a)}{s'(a)}f^-(a) - \frac{h^2(a-\epsilon)}{s'(a-\epsilon)} f^-(a-\epsilon) }{\epsilon}\\
&= \frac{2h(a)h^-(a)s'(a)-h^2(a)s''(a)}{(s'(a))^2}f^-(a) + \frac{h^2(a)}{s'(a)}(f^-)^-(a)\\
&=\frac{2s(a)(s'(a))^2-s^2(a)s''(a)}{(s'(a))^2}f^-(a) + \frac{s^2(a)}{s'(a)}(f^-)^-(a)
\end{split}
\ee
Next use \eqref{e:integral_f^+} for the interval $(a-\epsilon,a]$ to get
\be
\begin{split}
\int_{(a-\epsilon,a]}g(x)h^2(x)m'(x)\,dx &=\frac{h^2(a)}{s'(a)}f^+(a) - \frac{h^2(a-\epsilon)}{s'(a-\epsilon)} f^+(a-\epsilon) - \int_{(a-\epsilon,a]}s(a)\delta_a(dx)\\
&= \frac{h^2(a)}{s'(a)}f^+(a) - \frac{h^2(a-\epsilon)}{s'(a-\epsilon)} f^+(a-\epsilon) - s(a)f(a)
\end{split}
\ee
which implies
\be\label{e:f^+(a-)}
\left(f^+(a)-\lim_{\epsilon\downarrow 0}f^+(a-\epsilon)\right) = \frac{s'(a)}{h^2(a)}f(a)s(a) = \frac{s'(a)}{s(a)} f(a).
\ee
Next use \eqref{e:integral_f^+} for the interval $(a,a+\epsilon]$ to get
\be
\begin{split}
\int_{(a,a+\epsilon]}g(x)h^2(x)m'(x)\,dx &=\frac{h^2(a+\epsilon)}{s'(a+\epsilon)}f^+(a+\epsilon) - \frac{h^2(a)}{s'(a)} f^+(a) - \int_{(a,a+\epsilon]}s(a)\delta_a(dx)\\
&= \frac{h^2(a+\epsilon)}{s'(a+\epsilon)}f^+(a+\epsilon) - \frac{h^2(a)}{s'(a)} f^+(a)
\end{split}
\ee
which forces
\be
\begin{split}
g(a)h^2(a) m'(a)&= \lim_{\epsilon\downarrow 0} \frac{\frac{h^2(a+\epsilon)}{s'(a+\epsilon)}f^+(a+\epsilon) - \frac{h^2(a)}{s'(a)} f^+(a)}{\epsilon}\\
&=\frac{2h(a)h^+(a)s'(a)-h^2(a)s''(a)}{(s'(a))^2}f^+(a) + \frac{h^2(a)}{s'(a)}(f^+)^+(a)\\
&= \frac{-s^2(a)s''(a)}{(s'(a))^2}f^+(a) + \frac{s^2(a)}{s'(a)}(f^+)^+(a).\\
\end{split}
\ee
Suppose next that we have $g=\cG^h f$ for $f\in \cD(\cG^h)$. If $a\notin(u,v)$ we have
\begin{equation*}
\begin{split}
\int_{[u,v)}g(x)h^2(x)m'(x)dx &= \int_{(u,v]}g(x)h^2(x)m'(x)dx\\
&=  \int_{[u,v)}\left(\frac{h^2(x)}{s'(x)}f'(x)\right)'dx = \frac{h^2(v)}{s'(v)}f'(v)-\frac{h^2(u)}{s'(u)}f'(u).
\end{split}
\end{equation*}
Apply this, the fact that $h$, $g$, $s'$ are continuous, and \eqref{e:f^-(a+)} to get
\be
\begin{split}
\int_{[a,b)}g(x)h^2(x)m'(x)dx &= \lim_{\epsilon\downarrow 0}\int_{[a,a+\epsilon)}g(x)h^2(x)m'(x)dx+  \lim_{\epsilon\downarrow 0}\int_{[a+\epsilon,b)}g(x)h^2(x)m'(x)dx\\
&=  \lim_{\epsilon\downarrow 0}\int_{[a,a+\epsilon)}g(x)h^2(x)m'(x)dx +  \lim_{\epsilon\downarrow 0}\frac{h^2(b)}{s'(b)}f'(b)\\
&\quad -\lim_{\epsilon\downarrow 0}\frac{h^2(a+\epsilon)}{s'(a+\epsilon)}f'(a+\epsilon)\\
&=  \lim_{\epsilon\downarrow 0}\int_{[a,a+\epsilon)}g(x)h^2(x)m'(x)dx +  \frac{h^2(b)}{s'(b)}f^-(b)\\
&\quad -\lim_{\epsilon\downarrow 0}\frac{h^2(a+\epsilon)}{s'(a+\epsilon)}f^-(a+\epsilon)\\
&= 0 + \frac{h^2(b)}{s'(b)}f^-(b) - \frac{h^2(a)}{s'(a)}\left(f^-(a)+\frac{s'(a)}{s(a)}f(a)\right)\\
&= \frac{h^2(b)}{s'(b)}f^-(b) - \frac{h^2(a)}{s'(a)}f^-(a) - s(a)f(a)\\
&= \frac{h^2(b)}{s'(b)}f^-(b) - \frac{h^2(a)}{s'(a)}f^-(a) - \int_{[a,b)}f(x)s(a)\delta_a(dx).
\end{split}
\ee
Using the left continuity of $f^-$ one can also see that
\be
\begin{split}
\int_{[c,a)}g(x)h^2(x)m'(x)dx &=\lim_{\epsilon\downarrow 0}\int_{[c,a-\epsilon)}g(x)h^2(x)m'(x)dx \\
&\quad +\lim_{\epsilon\downarrow 0}\int_{[a-\epsilon,a)}g(x)h^2(x)m'(x)dx \\
&= \lim_{\epsilon\downarrow 0}\frac{h^2(a-\epsilon)}{s'(a-\epsilon)}f^-(a-\epsilon) - \frac{h^2(c)}{s'(c)}f'(c)\\
&= \frac{h^2(a)}{s'(a)}f^-(a) - \frac{h^2(c)}{s'(c)}f'(c).
\end{split}
\ee
Analogous arguments using \eqref{e:f^+(a-)} show that for $c<a$
\be
\begin{split}
\int_{(c,a]}g(x)h^2(x)m'(x)dx &=\lim_{\epsilon\downarrow 0}\int_{(c,a-\epsilon]}g(x)h^2(x)m'(x)dx \\
 &\quad +  \lim_{\epsilon\downarrow 0}\int_{(a-\epsilon,a]}g(x)h^2(x)m'(x)dx\\
&=  \lim_{\epsilon\downarrow 0}\frac{h^2(a-\epsilon)}{s'(a-\epsilon)}f^+(a-\epsilon)-\frac{h^2(c)}{s'(c)}f^+(c)\\
&= \frac{h^2(a)}{s'(a)}f^+(a) - \frac{h^2(c)}{s'(c)}f^+(c) - s(a)f(a)\\
&= \frac{h^2(a)}{s'(a)}f^+(a) - \frac{h^2(c)}{s'(c)}f^+(c) - \int_{(c,a]}f(x)s(a)\delta_a(dx)
\end{split}
\ee
and using the right continuity of $f^+$
\be
\begin{split}
\int_{(a,b]}g(x)h^2(x)m'(x)dx &= \lim_{\epsilon\downarrow 0}\int_{(a,a+\epsilon]}g(x)h^2(x)m'(x)dx\\
&\quad +  \lim_{\epsilon\downarrow 0}\int_{(a+\epsilon,b]}g(x)h^2(x)m'(x)dx\\
&= \frac{h^2(b)}{s'(b)}f^+(b)- \lim_{\epsilon\downarrow 0}\frac{h^2(a+\epsilon)}{s'(a+\epsilon)}f^+(a+\epsilon)\\
&= \frac{h^2(b)}{s'(b)}f^+(b)- \frac{h^2(a)}{s'(a)}f^+(a).
\end{split}
\ee
\end{proof}

\begin{example}
Consider conditioning Brownian motion with drift $-\mu$, $\mu > 0$, to be at $0$ at a large exponential time.
%
From Theorem \ref{t:cond_point} we get that $(X_t^h)_{t\geq 0}$ has generator
\[
\cG^hf(y) = \left\{
               \begin{array}{ll}
                 \frac{1}{2}f''(y)-\mu\sgn{(y)}f'(y), & \hbox{$y\neq 0$,} \\
		\frac{-s''(0)}{m'(0)(s'(0))^2}f^+(0) + \frac{1}{m'(0)s'(0)}(f^+)^+(0), & \hbox{$y= 0$,} \\         \end{array}
             \right.
\]
with domain 
\begin{equation*}
\begin{split}
&\cD(\cG^h)\\
&=\Bigg\{f\in C^2(-\infty,0)\cap C^2(0,\infty): f^+(0)-f^+(0-)=f^-(0+)- f^-(0)=\frac{s'(0)}{s(0)}f(0),\\
& \frac{2s(0)(s'(0))^2-s^2(0)s''(0)}{(s'(0))^2}f^-(0) + \frac{s^2(0)}{s'(0)}(f^-)^-(0)=\frac{-s^2(0)s''(0)}{(s'(0))^2}f^+(0) \\
&\quad + \frac{s^2(0)}{s'(0)}(f^+)^+(0) \Bigg\}.
\end{split}
\end{equation*}
Noting that $m'(x) = 2e^{-2\mu x}$ and $s(x)=e^{2\mu x}$ and $s(0)=\frac{1}{2\mu}$ straightforward computations yield 
\[
\cG_Z f(y) = \left\{
               \begin{array}{ll}
                 \frac{1}{2}f''(y)-\mu\sgn{(y)}f'(y), & \hbox{$y\neq 0$,} \\
		-\mu f^+(0) + \frac{1}{2}(f^+)^+(0), & \hbox{$y= 0$,} \\               \end{array}
             \right.
\]
with
\begin{equation*}
\begin{split}
\cD(\cG_Z)&=\Bigg\{f\in C^2(-\infty,0)\cap C^2(0,\infty): f^+(0)-f^+(0-)=f^-(0+)- f^-(0)=2\mu f(0),\\
& \mu f^-(0) + \frac{1}{2}(f^-)^-(0)=-\mu f^+(0) + \frac{1}{2}(f^+)^+(0) \Bigg\}.
\end{split}
\end{equation*}
\end{example}

\begin{example}
The solution to the SDE
\begin{equation}\label{e:pop2_}
dX_t = X_t(\mu-\kappa X_t)\,dt + \sigma X_t \,dW_t, t\geq 0.
\end{equation}
models a population living in one patch in which the individuals compete for resources. 
 assume that $\mu - \frac{\sigma^2}{2}<0$ so that $(X_t)_{t\geq0}$ is transient and $X_t\downarrow 0$ as $t\to \infty$ 
$\Pr^x$-almost surely for all $x\in (0,\infty)$. Note that if we start $(X_t)_{t\geq 0}$ at $x\in (0,\infty)$, 
 the process is almost surely positive for all $t\geq 0$.  See \cite{EHS14} for more details.
We study what we happens when we condition this diffusion for a point $a\in (0,\infty)$. Let $\mathcal {L}$ be the generator of $X$
\[
\mathcal{L} = (\mu x-\kappa x^2)\frac{d}{dx} + \frac{1}{2}\sigma^2 x^2 \frac{d^2}{dx^2}.
\]
The generator of $X^h$ is
\[
\mathcal{L}^h = \left(\mu x-\kappa x^2+\frac{\partial_x h(x)}{h(x)}\right)\frac{d}{dx} + \frac{1}{2}\sigma^2 x^2 \frac{d^2}{dx^2}
\]
with a suitable domain.

Making using of \eqref{e:r_0_ell} and \eqref{e:scale_speed_C^1} we get
\[
h(x) =
\begin{cases}  \int_0^x z^{-\frac{2\mu}{\sigma^2}} e^{\frac{2\kappa z}{\sigma^2}}\,dz, &\mbox{if } x \le a, \\
  1, & \mbox{if }x\geq a,
\end{cases}
\]
so the new drift is given by
\[
x \mapsto
\begin{cases} \mu x -\kappa x^2+ \frac{x^{-\frac{2\mu}{\sigma^2}} e^{\frac{2\kappa x}{\sigma^2}}}{\int_0^x z^{-\frac{2\mu}{\sigma^2}} e^{\frac{2\kappa z}{\sigma^2}}\,dz}, &\mbox{if } x \le a\\
  \mu x -\kappa x^2, & \mbox{if } x\geq a.
\end{cases}
\]
For $x$ small the new drift looks like
\[
x \mapsto \mu x -\kappa x^2 +\left(1 - \frac{2 \mu}{ \sigma^2}\right) \frac{1}{x}.
\]
\end{example}

\section{Another mode of conditioning}\label{s:another_conditioning}

There is another way to condition a Markov process to be in a fixed state at a large random time.
It is described in the following theorem.

\begin{theorem}\label{t:main_another_conditioning}
Let $X$ be a transient Borel right process on a Lusin space $E$ and suppose $a$ is a regular point. Assume furthermore that the resolvent $(R_\lambda)_{\lambda> 0}$ of $X$ has a density with respect to a measure $m$.
From excursion theory the amount of local time $X$ started in $a$ spends in $a$ is exponential with some rate $\nu$.
Take an independent exponential with rate $\lambda$ time $\zeta$ and
condition the process  $X$ to spend local time in $a$ that is at least $\zeta$ and kill the conditioned process when amount $\zeta$
of local time has been spent at $a$. Then, as we let $\lambda\downarrow 0$, we get the bang-bang process $X^b$ killed when the local time at $a$ exceeds an independent exponential with rate $\nu$.
\end{theorem}
\begin{proof}
By the competing exponentials result, it seems reasonable and can be shown by
excursion theory that if we condition on the amount of local time in $a$
being bigger than $\zeta$ and look at the conditioned process killed at $\zeta$
we get the bang-bang process $X^b$ killed when the local time at $a$ exceeds an
independent exponential with rate $\lambda + \nu$ and so letting $\lambda $ go
to zero we just get the bang-bang process $X^b$ killed when the local time
exceeds an exponential with rate $\nu$.  
\end{proof}

{\bf Acknowledgments.}  The authors thank Patrick Fitzsimmons and Paavo Salminen for helpful discussions.

\bibliographystyle{amsalpha}
\bibliography{asymptotic}
\end{document}